\documentclass[10pt, a4paper]{amsart}

\usepackage{framed}
\usepackage{amsmath, amssymb, amsfonts, amsthm}
\usepackage[margin=2.0cm]{geometry}
\usepackage[]{setspace}
\usepackage[usenames,dvipsnames]{color}
\usepackage[unicode]{hyperref}
\usepackage{comment}

\usepackage[all,matrix,arrow]{xy}
\usepackage{enumerate}
\usepackage{xcolor}
\usepackage[utf8]{inputenc}
\usepackage[T1]{fontenc}

\usepackage{pdflscape}
\usepackage{tikz}
\usetikzlibrary{fit,shapes,calc,arrows,through,intersections,arrows.meta,decorations.pathmorphing}
\usepackage{tkz-euclide}

\theoremstyle{plain}
\newtheorem{theorem}{Theorem}[section]
\newtheorem{proposition}[theorem]{Proposition}
\newtheorem{lemma}[theorem]{Lemma}
\newtheorem{corollary}[theorem]{Corollary}

\theoremstyle{definition}
\newtheorem{example}[theorem]{Example}
\newtheorem{definition}[theorem]{Definition}
\newtheorem{question}[theorem]{Question}

\theoremstyle{remark}
\newtheorem{remark}[theorem]{Remark}

\newcommand\bHen{\mathrm{H}^{e\prime}}
\newcommand\bHepi{\mathrm{H}^{e,\varpi}}
\newcommand\bHepiZ{\mathrm{H}^{e,\varpi,\mathbb{Z}}}
\newcommand\bHeZ{\mathrm{H}^{e,\mathbb{Z}}}
\newcommand\bHed{\mathrm{H}^{e,\mathrm{disc}}}

\renewcommand\int{\mathbf{int}}

\newcommand\Rfour{\hyperref[R4]{{\rm(R4)}}}

\renewcommand\Form{\mathrm{Form}}
\newcommand\Sent{\mathrm{Sent}}
\newcommand\FForm{\mathsf{Form}}

\numberwithin{equation}{section}
\makeatletter
\newcommand{\addresseshere}{%
  \enddoc@text\let\enddoc@text\relax
}
\makeatother

\newcommand{\new}[1]{#1}

\title{{Universal-existential theories of fields}}
\author{Sylvy Anscombe and Arno Fehm}
\address{Universit\'{e} Paris Cit\'{e} and Sorbonne Universit\'{e}, CNRS, IMJ-PRG, F-75013 Paris, France}
\email{sylvy.anscombe@imj-prg.fr}
\address{Institut f\"{u}r Algebra, Technische Universit\"{a}t Dresden, 01062 Dresden, Germany}
\email{arno.fehm@tu-dresden.de}

\begin{document}

~\vspace{-1.0cm}
\begin{center}
{\color{blue}\small Published in {\em Proceedings of the Edinburgh Mathematical Society}, 
\url{https://doi.org/10.1017/S0013091525100771} \\ 
This updated version contains a simplified proof of Proposition \ref{prop:Tred_ff_2}.}
\end{center}
\vspace{0.5cm}

\maketitle

\begin{center}
{\em Dedicated to the memory of Alexander Prestel (1941--2024)}
\end{center}

\begin{abstract}
We study universal-existential fragments of first-order theories of fields, in particular of function fields and of equicharacteristic henselian valued fields.
For example we discuss to what extent the theory of a field $k$ determines the universal-existential theories of the rational function field over $k$ and of the field of Laurent series over $k$,
and we find various many-one reductions between such fragments.
\end{abstract}

\section{Introduction}

\noindent
\new{The main application
of the celebrated paper \cite{AK3} by Ax and Kochen,
which initiated the study of the model theory of henselian valued fields,
is the decidability of the theory of the Laurent series fields $\mathbb{R}(\!(t)\!)$ and $\mathbb{C}(\!(t)\!)$,
answering a question by R.~Robinson.}
What they prove is  a bit more precisely the following:

\begin{theorem}[Ax--Kochen 1966]
Let $k,l$ be fields of characteristic zero.
If ${\rm Th}(k)={\rm Th}(l)$, then 
${\rm Th}(k(\!(t)\!),v_t,t)={\rm Th}(l(\!(t)\!),v_t,t)$, and if ${\rm Th}(k)$ is decidable,
then so is ${\rm Th}(k(\!(t)\!),v_t,t)$.
\end{theorem}

Here, we denote by ${\rm Th}(k(\!(t)\!),v_t,t)$ the theory of $k(\!(t)\!)$ with the $t$-adic valuation in some language of valued fields together with a constant symbol for the element $t$.
In spite of considerable effort, no full analogue of this has been obtained for fields of positive characteristic. Notably, the decidability of the local fields $\mathbb{F}_q(\!(t)\!)$ is a long-standing open problem,
despite for example \cite{Kuh01,Onay,Kartas}. 
However, by now we have at least results in this direction for certain fragments of the theory, most notably the existential theory ${\rm Th}_\exists$,
see for example \cite{DS03,AF16,Kartas_tame,ADF23}.
Some of these results are conditional on  consequences of resolution of singularities in positive characteristic, like the assumption \Rfour\ used below, which was introduced in \cite{ADF23}.

\begin{theorem}[\cite{AF16,ADF23}]
Let $k,l$ be fields.
\begin{enumerate}[$(a)$]
 \item If ${\rm Th}_\exists(k)\subseteq{\rm Th}_\exists(l)$, then
${\rm Th}_\exists(k(\!(t)\!),v_t)\subseteq{\rm Th}_\exists(l(\!(t)\!),v_t)$,
and if ${\rm Th}_\exists(k)$ is decidable, then so is ${\rm Th}_\exists(k(\!(t)\!),v_t)$.
\item Suppose \Rfour.
If ${\rm Th}_\exists(k)\subseteq{\rm Th}_\exists(l)$, then
${\rm Th}_\exists(k(\!(t)\!),v_t,t)\subseteq{\rm Th}_\exists(l(\!(t)\!),v_t,t)$,
and if ${\rm Th}_\exists(k)$ is decidable, then so is ${\rm Th}_\exists(k(\!(t)\!),v_t,t)$.
\end{enumerate}
\end{theorem}

The situation is somewhat comparable for rational function fields $k(t)$.
It is almost trivial to see that ${\rm Th}_\exists(k)\subseteq{\rm Th}_\exists(l)$
implies 
${\rm Th}_\exists(k(t))\subseteq{\rm Th}_\exists(l(t))$,
but both
the full theory ${\rm Th}(k(t))$
and the 
decidability of this existential theory
remain in general mysterious.
For example,
it is known that ${\rm Th}(k)={\rm Th}(l)$ does {\em not} imply ${\rm Th}(k(t))={\rm Th}(l(t))$,
and the decidability of ${\rm Th}(\mathbb{C}(t))$ is another long-standing open problem.
Also, while
it was shown that ${\rm Th}_\exists(\mathbb{F}_q(t),t)$ (the existential theory of $\mathbb{F}_q(t)$ in the language of rings together with a constant symbol for the element $t$)
is undecidable \cite{Pheidas},
the decidability of
${\rm Th}_\exists(\mathbb{F}_q(t))$
is open.

The aim of this work is to similarly study universal-existential theories of fields.
We push some of the results from existential to certain kinds of universal-existential theories,
and we study the relation between existential theories with parameters and universal-existential theories without parameters.
Besides the usual universal-existential theory
${\rm Th}_{\forall\exists}$ we are also interested in the
${\rm Th}_{\forall_n\exists}$-theories, 
by which we mean
sentences
of the form $\forall x_1,\dots,x_n \psi$ with $\psi$ existential.
As a variant in the Laurent series case $k(\!(t)\!)$ we also study those 
universal-existential sentences 
where the universal quantifiers run only over $k$, which we denote by
${\rm Th}_{\forall^{\mathbf{k}}\exists}$.
(Fully precise definitions and a more systematic treatment of such ``fragments'' can be found in Section \ref{sec:fragments}. There are subtleties which for this introduction we can ignore.)

We now explain some applications of our (more general) results to rational function fields and to Laurent series fields.
For rational function fields, in Section \ref{sec:ff} we prove in particular the following:

\begin{theorem}\label{thm:intro_functionfield}
Let $k,l$ be fields.    
\begin{enumerate}[$(a)$] 
\item If ${\rm Th}_\exists(k)\subseteq{\rm Th}_\exists(l)$,
then ${\rm Th}_\exists(k(t),t)\subseteq{\rm Th}_\exists(l(t),t)$.
\item If ${\rm Th}_{\forall_1\exists}(k)\subseteq{\rm Th}_{\forall_1\exists}(l)$, then
${\rm Th}_{\forall_1\exists}(k(t))\subseteq{\rm Th}_{\forall_1\exists}(l(t))$.
\item If ${\rm Th}_\exists(k(t),t)$
is decidable and
${\rm Th}_{\forall_1\exists}(k)$
is decidable, then
${\rm Th}_{\forall_1\exists}(k(t))$
is decidable.
\item If ${\rm Th}_{\forall_1\exists}(k(t))$ is decidable
and $k$ is existentially $\emptyset$-definable in $k(t)$,
then ${\rm Th}_\exists(k(t),t)$ is decidable.
\item Suppose $k,l$ are perfect of positive characteristic.
If
${\rm Th}_{\forall_1\exists}(k(t),t)\subseteq{\rm Th}_{\forall_1\exists}(l(t),t)$, then
${\rm Th}_{\forall\exists}(k(t),t)\subseteq{\rm Th}_{\forall\exists}(l(t),t)$,
and if ${\rm Th}_{\forall_1\exists}(k(t),t)$ is decidable, then so is
${\rm Th}_{\forall\exists}(k(t),t)$.
If
${\rm Th}_{\forall_2\exists}(k(t))\subseteq{\rm Th}_{\forall_2\exists}(l(t))$, then
${\rm Th}_{\forall\exists}(k(t))\subseteq{\rm Th}_{\forall\exists}(l(t))$,
and if ${\rm Th}_{\forall_2\exists}(k(t))$ is decidable, then so is
${\rm Th}_{\forall\exists}(k(t))$.
\end{enumerate}
\end{theorem}

For example, we deduce that ${\rm Th}_{\forall_1\exists}(k(t))$ is undecidable
for every perfect or large field $k$ of positive characteristic not containing the algebraic closure of the prime field (Corollary \ref{cor:kt_AE1_perfectlarge}).
We also deduce similar results for function fields of curves of genus at least $2$ (Corollary \ref{cor:ff_A1E_gg1}).

In Section \ref{sec:Laurent} we study various theories of henselian valued fields
and obtain the following consequences for Laurent series fields:

\begin{theorem}\label{thm:intro_Laurent}
Let $k,l$ be fields.    
\begin{enumerate}[$(a)$]
\item Suppose ${\rm char}(k)=0$.
If ${\rm Th}_{\forall\exists}(k)\subseteq{\rm Th}_{\forall\exists}(l)$, then
${\rm Th}_{\forall\exists}(k(\!(t)\!),v_t,t)\subseteq{\rm Th}_{\forall\exists}(l(\!(t)\!),v_t,t)$,
and if ${\rm Th}_{\forall\exists}(k)$ is decidable, then so is 
${\rm Th}_{\forall\exists}(k(\!(t)\!),v_t,t)$.
\item  Suppose \Rfour.
If ${\rm Th}_{\forall\exists}(k)\subseteq{\rm Th}_{\forall\exists}(l)$, then
${\rm Th}_{\forall\exists}(k(\!(t)\!),v_t,t)\subseteq{\rm Th}_{\forall\exists}(l(\!(t)\!),v_t,t)$,
and if ${\rm Th}_{\forall\exists}(k)$ is decidable, then so is 
${\rm Th}_{\forall^{\mathbf{k}}\exists}(k(\!(t)\!),v_t,t)$.
\item Suppose \Rfour\ or ${\rm char}(k)=0$.
If ${\rm Th}_{\forall_1\exists}(k)\subseteq{\rm Th}_{\forall_1\exists}(l)$, then
${\rm Th}_{\forall_1\exists}(k(\!(t)\!),v_t)\subseteq {\rm Th}_{\forall_1\exists}(k(\!(t)\!),v_t)$,
and if ${\rm Th}_{\forall_1\exists}(k)$ is decidable, then so is
${\rm Th}_{\forall_1\exists}(k(\!(t)\!),v_t)$.
\item If ${\rm Th}_{\forall_1\exists}(k(\!(t)\!),v_t)\subseteq {\rm Th}_{\forall_1\exists}(l(\!(t)\!),v_t)$, then
${\rm Th}_{\exists}(k(\!(t)\!),v_t,t)\subseteq
{\rm Th}_{\exists}(l(\!(t)\!),v_t,t)$,
and if ${\rm Th}_{\forall_1\exists}(k(\!(t)\!),v_t)$ is decidable, then so is ${\rm Th}_{\exists}(k(\!(t)\!),v_t,t)$.
\item Suppose that $k,l$ are finite.
If
${\rm Th}_{\exists}(k(\!(t)\!),v_t,t)\subseteq
{\rm Th}_{\exists}(l(\!(t)\!),v_t,t)$,
then ${\rm Th}_{\forall_1\exists}(k(\!(t)\!),v_t)\subseteq {\rm Th}_{\forall_1\exists}(l(\!(t)\!),v_t)$,
and if ${\rm Th}_{\exists}(k(\!(t)\!),v_t,t)$ is decidable, then so is
${\rm Th}_{\forall_1\exists}(k(\!(t)\!),v_t)$.
\item Suppose $k,l$ are perfect of positive characteristic.
If
${\rm Th}_{\forall_1\exists}(k(\!(t)\!),v_t,t)\subseteq{\rm Th}_{\forall_1\exists}(l(\!(t)\!),v_t,t)$, then \linebreak
${\rm Th}_{\forall\exists}(k(\!(t)\!),v_t,t)\subseteq{\rm Th}_{\forall\exists}(k(\!(t)\!),v_t,t)$,
and if ${\rm Th}_{\forall_1\exists}(k(\!(t)\!),v_t,t)$ is decidable, then so is
${\rm Th}_{\forall\exists}(k(\!(t)\!),v_t,t)$.
If
${\rm Th}_{\forall_2\exists}(k(\!(t)\!),v_t)\subseteq{\rm Th}_{\forall_2\exists}(l(\!(t)\!),v_t)$, then
${\rm Th}_{\forall\exists}(k(\!(t)\!),v_t)\subseteq{\rm Th}_{\forall\exists}(l(\!(t)\!),v_t)$,
and if ${\rm Th}_{\forall_2\exists}(k(\!(t)\!),v_t)$ is decidable, then so is
${\rm Th}_{\forall\exists}(k(\!(t)\!),v_t)$.
\end{enumerate}
\end{theorem}

All results stated here for decidability are in fact consequences of many-one reductions that we obtain.
For $k(\!(t)\!)$ we summarized these reductions in Figure~\ref{diag:Tred} on p.~\pageref{diag:Tred}.

After the already announced treatment of fragments in Section \ref{sec:fragments},
we introduce some general tools for studying ${\rm Th}_{\forall\exists}$ and
${\rm Th}_{\forall_n\exists}$ in Section \ref{sec:criteria}.
In Section \ref{sec:coding}
we explain the positive characteristic 
phenomenon that leads to
Theorem \ref{thm:intro_functionfield}$(e)$ and
Theorem \ref{thm:intro_Laurent}$(f)$.
In sections \ref{sec:ff} and \ref{sec:Laurent}
we then prove the specific results for function fields and for henselian valued fields.
We made an effort to phrase most of our results in a generality that should make them applicable beyond the examples presented in this introduction.
\new{Appendix \ref{sec:appendix} makes precise the computability theoretic notions that we use.}

\section{Fragments}
\label{sec:fragments}

\noindent
Let $\mathfrak{L}$ be a language.
\new{The reader interested only in the theorems stated in the introduction
may suppose that $\mathfrak{L}$
is the language of rings
or valued fields,
possibly expanded by constants.}
As usual, we denote by $\Form(\mathfrak{L})$ the set of $\mathfrak{L}$-formulas
and by $\Sent(\mathfrak{L})\subseteq\Form(\mathfrak{L})$ the set of $\mathfrak{L}$-sentences,
i.e.~the set of $\varphi\in\Form(\mathfrak{L})$ with free variables ${\rm Var}(\varphi)$ empty,
and for an $\mathfrak{L}$-structure $M$
we denote by ${\rm Th}(M)$ the $\mathfrak{L}$-theory of $M$.
An {\em $\mathfrak{L}$-fragment}%
\footnote{In \cite{AF23} an `$\mathfrak{L}$-fragment' is a set of {\em $\mathfrak{L}$-sentences} closed under conjunction and disjunction, and containing $\top$ and $\bot$,
so it is precisely an $\mathfrak{L}$-fragment in our sense that consists of sentences.
There is also a notion of `fragment' introduced in \cite[Definition 2.3]{GHS}
that is similar to --- but strictly stronger than --- our notion of `$\mathfrak{L}$-fragment'.}
is a set $F$ of $\mathfrak{L}$-formulas 
that contains $\top$ and $\bot$ 
\new{(symbols for the atomic sentences {\em verum} and {\em falsum})},
is closed under conjunction and disjunction,
and is closed under free variable substitution,
i.e.~if $\varphi(x_1,\dots,x_n)\in F$, then also $\varphi(y_1,\dots,y_n)\in F$ where $y_1,\dots,y_n$ are arbitrary variables (not necessarily distinct).
For $F\subseteq{\rm Form}(\mathfrak{L})$
we define
$\Form_F(\mathfrak{L}):=F$,
$\Sent_F(\mathfrak{L}):=\Sent(\mathfrak{L})\cap F$,
${\rm Th}_F(M):={\rm Th}(M)\cap F$,
and we write $M\equiv_FM'$ for ${\rm Th}_F(M)={\rm Th}_F(M')$.

To precisely state the uniform aspects of some of our results, we take a functorial approach to fragments:
Namely, we can view $\Form$ as a covariant functor 
$\FForm$
from the category $\mathbf{Lang}$ of languages (with inclusions of languages as morphisms)
to the category $\mathbf{Sets}$ of sets (with inclusions as morphisms).
%
\new{%
The reader not interested in questions of uniformity
can continue to fix their favourite language $\mathfrak{L}$
and proceed directly to Definition \ref{def:fragments},
which is then to be read in light of Remark~\ref{rem:X}.
}

\begin{definition}\label{def:functors}
A {\em functor of formulas}
is a subfunctor $\mathsf{F}$ 
of
the restriction of $\FForm$ to a full subcategory $\mathbf{L}$ of $\mathbf{Lang}$,
and we call it 
a {\em fragment} if 
$\mathsf{F}(\mathfrak{L})$ is an $\mathfrak{L}$-fragment for every $\mathfrak{L}\in\mathbf{L}$,
and in addition
$\mathsf{F}(\mathfrak{L}_2)\cap{\rm Form}(\mathfrak{L}_1)=\mathsf{F}(\mathfrak{L}_1)$
for every $\mathfrak{L}_1\subseteq\mathfrak{L}_2$ in $\mathbf{L}$.
For a functor of formulas $\mathsf{F}$
we also write $\Form_{\mathsf{F}}(\mathfrak{L})$ for ${\rm Form}_{\mathsf{F}(\mathfrak{L})}(\mathfrak{L})=\mathsf{F}(\mathfrak{L})$,
and
$\Sent_{\mathsf{F}}(\mathfrak{L})$ for ${\rm Sent}_{\mathsf{F}(\mathfrak{L})}(\mathfrak{L})=\Form_{\mathsf{F}}(\mathfrak{L})\cap\Sent(\mathfrak{L})$,
${\rm Th}_\mathsf{F}(M)$ for ${\rm Th}_{\mathsf{F}(\mathfrak{L})}(M)={\rm Th}(M)\cap\Sent_\mathsf{F}(\mathfrak{L})$,
and $M\equiv_\mathsf{F}M'$ for ${\rm Th}_\mathsf{F}(M)={\rm Th}_\mathsf{F}(M')$.
Given functors of formulas $\mathsf{F}_1,\mathsf{F}_2$ with domain $\mathbf{L}$,
we say that a family of maps
$\tau_\mathfrak{L}\colon\mathsf{F}_1(\mathfrak{L})\rightarrow\mathsf{F}_2(\mathfrak{L})$ is
{\em uniform} (in $\mathfrak{L}$)
if it is the family of components of a natural transformation $\tau\colon\mathsf{F}_1\rightarrow\mathsf{F}_2$,
\end{definition}

\new{In the following, 
we will speak of computability of languages, sets of formulas, and functors of formulas, 
and we mean this in the usual or obvious sense.
The reader interested in precise definitions can consult Appendix \ref{sec:appendix}.
}

\begin{remark}\label{rem:existential}
Although expressions like ``existential'', ``universal-existential'', ``existential with two quantifiers'' etc.~occur frequently in the literature, it seems that it is often not clear what is meant precisely, and sometimes the different possible meanings differ in an essential way.
For example
there are already very different things meant by "existential formula":
A prenex existential formula, a positive boolean combination of such, a formula in which each existential quantifier is in the scope of an even number of negations and each universal quantifier in the scope of an odd number, or a formula logically equivalent to a prenex existential formula.
Of course, here the question whether two structures satisfy the same ``existential'' sentences does not depend on the precise definition, but the option ``logically equivalent to a prenex existential formula'' presents computability issues, cf.~\cite[Remark~3.5]{AF23}.
The problem becomes more pressing when one counts quantifiers. While it is still true that two structures satisfy the same prenex existential formulas with say two quantifiers if and only if they satisfy the same positive boolean combination of such,
the consequences of a (possibly incomplete) theory that are prenex existential with two quantifiers does not determine
the set of consequences that are positive boolean combinations of such.
We will therefore now give precise definitions of the fragments we will use.
\end{remark}

\begin{definition}\label{def:fragments}
\new{We define a functor of formulas $\mathsf{F}_{0}$ by letting
$\mathsf{F}_{0}(\mathfrak{L})$ be the set of quantifier-free $\mathfrak{L}$-formulas.}
For $Q\in\{\forall,\exists\}$ and a functor of formulas $\mathsf{F}$ 
we inductively define functors of formulas
$Q_n[\mathsf{F}]$, $Q_n\mathsf{F}$,  $Q^n\mathsf{F}$ and $Q\mathsf{F}$ by
\begin{enumerate}[$(i)$]
\item $Q_{1}[\mathsf{F}](\mathfrak{L}):=\{Qx\psi : \psi\in\mathsf{F}(\mathfrak{L}),x\mbox{ a variable}\}\cup\mathsf{F}(\mathfrak{L})$,
\item
$Q_{n+1}[\mathsf{F}]:=Q_{1}[Q_{n}[\mathsf{F}]]$,
\item
$Q_{1}\mathsf{F}(\mathfrak{L})$ is the smallest $\mathfrak{L}$-fragment containing $Q_{1}[\mathsf{F}](\mathfrak{L})$,
\item
$Q_{n+1}\mathsf{F}:=Q_{1}Q_{n}[\mathsf{F}]$,
\item
$Q^{1}\mathsf{F}:=Q_1\mathsf{F}$,
\item
$Q^{n+1}\mathsf{F}:=Q_{1}Q^{n}\mathsf{F}$, and
\item
$Q\mathsf{F} 
:=\bigcup_nQ^n\mathsf{F}$.
\end{enumerate}
We abbreviate
$\exists_n\mathsf{F}_0$, 
$\forall_n\mathsf{F}_0$,
$\exists^n\mathsf{F}_0$, 
$\forall^n\mathsf{F}_0$,
$\exists\mathsf{F}_0$, 
$\forall\mathsf{F}_0$,
by
$\exists_n$, $\forall_n$, $\exists^n$, $\forall^n$, $\exists$, $\forall$. 
\end{definition}

\new{
\begin{remark}\label{rem:X}
If one is interested in only one language $\mathfrak{L}$,
then in Definition~\ref{def:fragments} 
one may safely ignore the framework of functors,
and read the above as an inductive definition of sets
$\mathsf{F}(\mathfrak{L})$
of $\mathfrak{L}$-formulas,
in terms of combinations $\mathsf{F}$ of the quantifiers $\exists$ and $\forall$, and some punctuation.
This suffices to precisely define, for example,
the set of existential $\mathfrak{L}$-formulas
${\rm Form}_\exists(\mathfrak{L})=\exists(\mathfrak{L})$,
and similarly
${\rm Form}_\mathsf{F}(\mathfrak{L})=\mathsf{F}(\mathfrak{L})$
for almost any combination $\mathsf{F}$ of quantifiers
that appears in this paper,
including $\exists_{n}$, $\forall_{1}\exists$, $\forall_{2}\exists$, and $\forall\exists$.
\end{remark}
}

\begin{remark}
\new{All of $\mathsf{F}_{0}$, $Q_{n}\mathsf{F}$, $Q^{n}\mathsf{F}$ and $Q\mathsf{F}$ 
are fragments,
$\mathsf{F}_{0}$ is computable (in the sense of Appendix \ref{sec:appendix}),}
and if $\mathsf{F}$ is computable , then so are
$Q_n[\mathsf{F}]$, $Q_n\mathsf{F}$, $Q^n\mathsf{F}$, and
$Q\mathsf{F}$.
This applies in particular to the fragments
$\exists_n$, $\forall_n$, $\exists^n$, $\forall^n$, $\exists$, $\forall$. 
For an illustration of the difference between $Q_n$ and $Q^n$ 
see Example \ref{ex:QnQn}.
\end{remark}

\begin{remark}\label{rem:prnx}
It is well known that every formula can be turned into an equivalent formula in prenex normal form
(a sequence of quantifiers followed by a quantifier-free formula), and that this can be done in a computable way, as long as the presentation of the language is computable. 
We can also do this {\em relative to a fragment} $\mathsf{F}$ extending $\mathsf{F}_0$.
Let $\mathsf{F}'$ be the smallest fragment extending $\mathsf{F}$ such that each $\mathsf{F}'(\mathfrak{L})$ is closed under negation.
Now, for $\varphi\in\Form(\mathfrak{L})$, we define
${\rm prnx}_{\mathsf{F},\mathfrak{L}}(\varphi)$ inductively as follows:
If $\varphi\in\Form_{\mathsf{F}'}(\mathfrak{L})$
let ${\rm prnx}_{\mathsf{F},\mathfrak{L}}(\varphi)=\varphi$;
otherwise, 
if $\varphi=Q x\psi$ for $Q\in\{\forall,\exists\}$ let
${\rm prnx}_{\mathsf{F},\mathfrak{L}}(\varphi)=Q x\;{\rm prnx}_{\mathsf{F},\mathfrak{L}}(\psi)$;
if $\varphi=\neg\psi$ let $n$ be minimal such that
${\rm prnx}_{\mathsf{F},\mathfrak{L}}(\psi)=Q_1 x_1\dots Q_n x_n\eta$
with $Q_i\in\{\forall,\exists\}$ and $\eta\in\Form_{\mathsf{F}'}(\mathfrak{L})$ and let
${\rm prnx}_{\mathsf{F},\mathfrak{L}}(\varphi)=Q_1' x_1\dots Q_n' x_n\neg\eta$, where $Q_i'=\forall$ if $Q_i=\exists$ and vice versa;
and if $\varphi=\psi_1\square\psi_2$ where $\square\in\{\wedge,\vee\}$
write $\psi_i=Q_{i,1}x_{i,1}\dots Q_{i,n_i}x_{i,n_i}\eta_i$ with $\eta_i\in\Form_{\mathsf{F}'}(\mathfrak{L})$ and $n_i$ minimal,
replace each $x_{2,j}$ that occurs among the $x_{1,j'}$ or free in $\eta_1$
by the next unused variable, and let
${\rm prnx}_{\mathsf{F},\mathfrak{L}}(\varphi)=Q_{1,1}x_{1,1}\dots Q_{1,n_1}x_{i,n_1}Q_{2,1}x_{2,1}\dots Q_{2,n_2}x_{2,n_2}(\eta_1\square\eta_2)$.
Note that $\mathrm{prnx}_{\mathsf{F},\mathfrak{L}}$ is uniform in $\mathfrak{L}$,
by the assumption
$\mathsf{F}(\mathfrak{L}_{1})=\Form(\mathfrak{L}_{1})\cap\mathsf{F}(\mathfrak{L}_{2})$
whenever $\mathfrak{L}_{1}\subseteq\mathfrak{L}_{2}$,
which is why we will sometimes omit the $\mathfrak{L}$ from the notation.
The usual prenex normal form of $\varphi$ is then
${\rm prnx}_{\mathsf{F}_0}(\varphi)$,
for which we also write ${\rm prnx}(\varphi)$.
For example, if $\varphi\in\Form_\exists(\mathfrak{L})$, then
${\rm prnx}(\varphi)\in\Form_{\exists_n}(\mathfrak{L})$ for some $n$,
and if 
$\varphi\in\Form_{\forall\mathsf{F}}(\mathfrak{L})$ for some fragment $\mathsf{F}$,
then ${\rm prnx}_{\mathsf{F}}(\varphi)\in{\rm Form}_{\forall_n[\mathsf{F}]}(\mathfrak{L})$ for some $n$.
If $\mathsf{F}$ is computable,
then ${\rm prnx}_{\mathsf{F},\mathfrak{L}}$ is  computable uniformly in $\mathfrak{L}$.
\end{remark}

\section{General model theoretic criteria}
\label{sec:criteria}

\noindent
Let $\mathfrak{L}$ be a language.
It is well known that an $\mathfrak{L}$-sentence is existential if and only if it is preserved under embeddings of $\mathfrak{L}$-structures,
and that 
$\mathfrak{L}$-structures $M$ and $N$ satisfy
${\rm Th}_\exists(M)\subseteq{\rm Th}_\exists(N)$
if and only if there exists an embedding
of $M$ into an elementary extension of $N$.
A related criterion was given in
\cite[Lemma~3.23(b)]{AF23}:
in particular, for $n\geq1$, we have
$\mathrm{Th}_{\exists_{n}}(M)\subseteq\mathrm{Th}_{\exists_{n}}(N)$ if and only if $\mathrm{Th}_{\exists}(M')\subseteq\mathrm{Th}_{\exists}(N)$ for every substructure $M'\subseteq M$ generated by at most $n$ elements.
We now discuss similar criteria for the fragments $\forall\exists$
and $\forall_{n}\exists$.
As usual $M\preccurlyeq_\exists N$ denotes that $M$ is existentially closed in $N$.
We write ${\rm diag}$ for the atomic diagram and ${\rm eldiag}$ for the elementary (complete) diagram of a structure.

\begin{lemma}\label{lem:AE_abstract}
For $\mathfrak{L}$-structures
$M$ and $N$ the following are equivalent:
\begin{enumerate}[$(a)$]
\item $\mathrm{Th}_{\forall\exists}(M)\subseteq\mathrm{Th}_{\forall\exists}(N)$
\item For every $\max\{|N|,|\mathfrak{L}|,\aleph_0\}$-saturated $M^*\equiv_{\forall\exists} M$ there exists an existentially closed embedding
$N\rightarrow M^*$.
\item There exists some $M^{*}\equiv M$ with an existentially closed embedding $N\rightarrow M^{*}$.
\end{enumerate}
\end{lemma}

\begin{proof}
The implication $(b)\Rightarrow(c)$ is trivial,
and $(c)\Rightarrow(a)$ is immediate from the definitions:
If $N\preccurlyeq_\exists M^*\equiv_{\forall\exists} M$ and $\varphi\in{\rm Th}_{\forall\exists}(M)$, then without loss of generality $\varphi$ is of the form $\forall\underline{x}\psi(\underline{x})$ with $\psi\in{\rm Form}_\exists(\mathfrak{L})$ (Remark \ref{rem:prnx}).
Now $M^*\equiv_{\forall\exists} M\models\varphi$ implies that $M^*\models\psi(\underline{a})$
for every $\underline{a}\in N^n$,
and therefore $N\models\psi(\underline{a})$
by the assumption $N\preccurlyeq_\exists M^*$.

To prove $(b)$ assuming $(a)$,
first note that we can assume without loss of generality that
$M^*\equiv M$.
Now
\cite[Proposition 5.2.2]{CK}
gives elementary extensions $M\preccurlyeq M'$ and $N\preccurlyeq N'$ such that
$N\leq M'\leq N'$.
By the Löwenheim--Skolem theorem 
there exists $N\leq M''\preccurlyeq M'$
of cardinality
$|M''|=\kappa:=\max\{|N|,|\mathfrak{L}|,\aleph_{0}\}$,
in particular $N\preccurlyeq_\exists M''$.
As the $\kappa$-saturated $M^*$ is $\kappa^+$-universal \cite[Theorem 8.1.6]{Hodges},
there is an elementary embedding $M''\rightarrow M^*$.
\end{proof}

\begin{remark}
We focus on universal-existential theories but could instead talk about existential-universal theories, as
(a) of Lemma \ref{lem:AE_abstract} is trivially equivalent to
\begin{enumerate}
    \item[\textbullet] $\mathrm{Th}_{\exists\forall}(M)\supseteq\mathrm{Th}_{\exists\forall}(N)$
\end{enumerate}
We will not make use of it but point out that (a)-(c) of Lemma \ref{lem:AE_abstract} are also equivalent to
\begin{enumerate}
    \item[\textbullet] There exists a chain $(M_i)_{i\in I}$ of $\mathfrak{L}$-structures
    with $\varinjlim_{i\in I}M_i\equiv N$
    and $M_i\equiv M$ for every $i\in I$.
\end{enumerate}
as follows for example from
\cite[Theorem 3.1.8]{TZ} (with $T_1={\rm Th}(M)$, $T_2={\rm Th}(N)$).
\end{remark}

\begin{lemma}\label{lem:AnE_abstract}
Let $n\in\mathbb{N}$.
For two
$\mathfrak{L}$-structures
$M$ and $N$,
the following are equivalent:
\begin{enumerate}[$(a)$]
\item
$\mathrm{Th}_{\forall_{n}\exists}(M)\subseteq\mathrm{Th}_{\forall_{n}\exists}(N)$
\item
For every
$N^{*}\succcurlyeq N$
and every $\underline{b}\in (N^{*})^n$,
there exist $N_{\underline{b}}\succcurlyeq N^{*}$,
$M_{\underline{b}}\succcurlyeq M$,
and an $\mathfrak{L}$-embedding 
$f_{\underline{b}}\colon M_{\underline{b}}\rightarrow N_{\underline{b}}$
with $b_1,\dots,b_n\in f_{\underline{b}}(M_{\underline{b}})$.
\item
For every $N^{*}\succcurlyeq N$
there exists $N^{**}\succcurlyeq N^{*}$
such that for every ${\underline{b}}\in (N^{*})^n$
there exist $M_{\underline{b}}\succcurlyeq M$ and an $\mathfrak{L}$-embedding 
$f_{\underline{b}}\colon M_{\underline{b}}\rightarrow N^{**}$
with $b_1,\dots,b_n\in f_{\underline{b}}(M_{\underline{b}})$.
\item
There exists $N^{*}\succcurlyeq N$
such that for every ${\underline{b}}\in(N^{*})^n$
there exist $M_{\underline{b}}\succcurlyeq M$ and an $\mathfrak{L}$-embedding 
$f_{\underline{b}}\colon M_{\underline{b}}\rightarrow N^{*}$
with $b_1,\dots,b_n\in f_{\underline{b}}(M_{\underline{b}})$.
\item
There exists $N^{*}\succcurlyeq N$
such that for every ${\underline{b}}\in(N^{*})^n$
there exist $M_{\underline{b}}\equiv_{\forall_{n}\exists}M$ and an $\mathfrak{L}$-embedding 
$f_{\underline{b}}\colon M_{\underline{b}}\rightarrow N^{*}$
with $b_1,\dots,b_n\in f_{\underline{b}}(M_{\underline{b}})$.
\item
For every ${\underline{b}}\in N^n$,
there exist $N_{\underline{b}}\succcurlyeq_\exists N$,
$M_{\underline{b}}\equiv_{\forall_n\exists} M$,
and an $\mathfrak{L}$-embedding 
$f_{\underline{b}}\colon M_{\underline{b}}\rightarrow N_{\underline{b}}$
with $b_1,\dots,b_n\in f_{\underline{b}}(M_{\underline{b}})$.
\end{enumerate}
\end{lemma}

\begin{proof}
$(a)\Rightarrow(b)$.
Without loss of generality $N=N^{*}$.
Let ${\underline{b}}\in N^n$
and let $\underline{x}$ denote the tuple of variables $(x_1,\dots,x_n)$.
Write 
$$
p(\underline{x})=\{\forall y_1,\dots,y_m\;\psi(\underline{x},\underline{y})\mid m\in\mathbb{N},N\models\forall\underline{y}\;\psi(\underline{b},\underline{y}),\psi\in{\rm Form}_{\mathsf{F}_0}(\mathfrak{L})\}.
$$
This is the universal type of ${\underline{b}}$ in $N$.
We claim that $p(\underline{x})\cup {\rm eldiag}(M)$ is finitely satisfiable:
a conjunction of finitely many elements of $p(\underline{x})$ is logically equivalent to a universal formula in free variables $\underline{x}$, that is again in $p(\underline{x})$,
say
$\forall\underline{y}\;\psi(\underline{x},\underline{y})$.
Since ${\underline{b}}\in N^n$, we have
$N\models\exists\underline{x}\forall\underline{y}\;\psi(\underline{x},\underline{y})$.
By (a) this implies
$M\models\exists\underline{x}\forall\underline{y}\;\psi(\underline{x},\underline{y})$,
which proves the claim.
So by the compactness theorem there is an elementary extension $M_{\underline{b}}\succcurlyeq M$ and a tuple $\underline{a}_{\underline{b}}\in M_{\underline{b}}^n$ such that
$M_{\underline{b}}\models p(\underline{a}_{\underline{b}})$.

We now consider the 
language $\mathfrak{L}'=\mathfrak{L}(\underline{c})$ where $\underline{c}=(c_1,\dots,c_n)$ is an $n$-tuple of new constant symbols,
and the $\mathfrak{L}(\underline{c})$-structures
$(M_{\underline{b}},\underline{a}_{\underline{b}})$ and
$(N,\underline{b})$.
The $\mathfrak{L}'(M_{\underline{b}}\cup N)$-theory
$T=\mathrm{diag}(M_{\underline{b}},\underline{a}_{\underline{b}})\cup\mathrm{eldiag}(N,{\underline{b}})$
is finitely satisfiable:
A conjunction of finitely many elements from
$\mathrm{diag}(M_{\underline{b}},\underline{a}_{\underline{b}})$ is logically equivalent to
a sentence of the form
$\varphi(\underline{c},\underline{d})$ for a quantifier-free $\mathfrak{L}$-formula $\varphi(\underline{x},\underline{y})$ and $d_1,\dots,d_m\in M_{\underline{b}}$.
Then $M_{\underline{b}}\not\models\forall\underline{y}\neg\varphi(\underline{a}_{\underline{b}},\underline{y})$
and therefore $N\models\exists\underline{y}\varphi(\underline{b},\underline{y})$,
so $(N,\underline{b})$ can be expanded to a model
of ${\rm eldiag}(N,\underline{b})\cup\{\varphi(\underline{c},\underline{d})\}$.
So by the compactness theorem, $T$ has a model $N_{\underline{b}}$,
which is then an elementary extension of $N$,
and the $\mathfrak{L}$-embedding $f_{\underline{b}}\colon M_{\underline{b}}\rightarrow N_{\underline{b}}$
with $f_{\underline{b}}(\underline{a}_{\underline{b}})=\underline{b}$
comes as usual from the method of diagrams.

$(b)\Rightarrow(c)$.
For each $\underline{b}\in(N^*)^n$ fix $N_{\underline{b}}$, $M_{\underline{b}}$ and $f_{\underline{b}}$ as given by $(b)$.
Choose a cardinal number $\kappa$ with $\kappa>|N_{\underline{b}}|$ for every $\underline{b}\in(N^*)^n$.
Let $N^{**}\succcurlyeq N^{*}$ be $\kappa$-saturated,
in particular $\kappa^+$-universal,
so that for every $\underline{b}\in(N^*)^n$
there is an elementary $\mathfrak{L}(N^*)$-embedding $g_{\underline{b}}\colon N_{\underline{b}}\rightarrow N^{**}$.
Then $g_{\underline{b}}\circ f_{\underline{b}}\colon M_{\underline{b}}\rightarrow N^{**}$ is as required for (c).

$(c)\Rightarrow(d)$.
Write $M_{0}=M$ and $N_{0}=N$.
We recursively construct
an increasing elementary chain
$(N_{i})_{i\in\omega}$, using $(c)$,
such that
for each $i$ and ${\underline{b}}\in(N_i)^n$ there exists $M_{\underline{b}}\succcurlyeq M$, $f_{\underline{b}}\colon M_{\underline{b}}\rightarrow N_{i+1}$ with $b_1,\dots,b_n\in f_{\underline{b}}(M_{\underline{b}})$.
Then $N^*:=\varinjlim_i N_i$ satisfies the claim.

$(d)\Rightarrow(e)\Rightarrow(f)$.
Trivial.

$(f)\Rightarrow(a)$.
Let $\varphi\in {\rm Th}_{\forall_n\exists}(M)$.
Without loss of generality, $\varphi=\forall x_1,\dots,x_n\exists\underline{y}\;\psi(\underline{x},\underline{y})$ with $\psi$ quantifier-free (Remarks \ref{rem:existential} and \ref{rem:prnx}).
Let $\underline{b}\in N^n$ and let $N_{\underline{b}}$, $M_{\underline{b}}$ and $f_{\underline{b}}\colon M_{\underline{b}}\rightarrow N_{\underline{b}}$ be given as in $(f)$, in particular 
there exists $\underline{a}_{\underline{b}}\in (M_{\underline{b}})^n$ with
$f_{\underline{b}}(\underline{a}_{\underline{b}})=\underline{b}$.
Since $M_{\underline{b}}\equiv_{\forall_n\exists} M$ and $M\models\varphi$, also $M_{\underline{b}}\models\varphi$,
so there exists $\underline{c}\in(M_{\underline{b}})^m$ such that
$M_{\underline{b}}\models\psi(\underline{a}_{\underline{b}},\underline{c})$.
Applying $f_{\underline{b}}$ we get
$N_{\underline{b}}\models\psi(\underline{b},f_{\underline{b}}(\underline{c}))$,
in particular
$N_{\underline{b}}\models\exists\underline{y}\psi(\underline{b},\underline{y})$.
Since $\underline{b}\in N^n$ and $N\preccurlyeq_{\exists}N_{\underline{b}}$,
also $N\models\exists\underline{y}\psi(\underline{b},\underline{y})$.
Thus, $N\models\varphi$.
\end{proof}

\begin{corollary}\label{cor:AE}
For $\mathfrak{L}$-structures
$M$ and $N$ the following are equivalent:
\begin{enumerate}[$(a)$]
\item $\mathrm{Th}_{\forall\exists}(M)\subseteq\mathrm{Th}_{\forall\exists}(N)$
\item There exists
$N^{*}\succcurlyeq N$
such that for
every finite subset $B\subseteq N^{*}$
there exist $M_B\succcurlyeq M$
and an $\mathfrak{L}$-embedding 
$f_{B}\colon M_{B}\rightarrow N^*$
with $B\subseteq f_{B}(M_{B})$.
\end{enumerate}
\end{corollary}

\begin{proof}
$(a)\Rightarrow(b)$: By Lemma \ref{lem:AnE_abstract}$(d)$
we can construct an elementary chain
$N=N_1\preccurlyeq N_2\preccurlyeq\dots$
such that for each $n$
and $B\subseteq N_n$ with $|B|\leq n$
there exists $M_B\succcurlyeq M$ and an $\mathfrak{L}$-embedding $f_B\colon M_B\rightarrow N_{n+1}$ with $B\subseteq f_B(M_B)$. Thus $N^*:=\varinjlim_n N_n$ satisfies $(b)$.

$(b)\Rightarrow(a)$: Condition $(b)$ implies Lemma \ref{lem:AnE_abstract}$(d)$ for every $n$, so
$\mathrm{Th}_{\forall_{n}\exists}(M)\subseteq\mathrm{Th}_{\forall_{n}\exists}(N)$
for every $n$. Since every $\varphi\in{\rm Form}_{\forall\exists}(\mathfrak{L})$
is equivalent to
${\rm prnx}_{\exists,\mathfrak{L}}(\varphi)\in{\rm Form}_{\forall_n\exists}(\mathfrak{L})$ for some $n$,
it follows that
$\mathrm{Th}_{\forall\exists}(M)\subseteq\mathrm{Th}_{\forall\exists}(N)$.
\end{proof}

\begin{example}\label{ex:QnQn}
We will not make much use of the fragments $\forall^n\exists$, but would nevertheless like to now discuss their relation to the fragments $\forall_n\exists$ that we indeed will work with.
By definition, 
${\rm Th}_{\forall_1\exists}(M)={\rm Th}_{\forall^1\exists}(M)$ for every $\mathfrak{L}$-structure $M$.
The following example in the language $\mathfrak{L}=\{E\}$ of graphs shows that
${\rm Th}_{\forall_2\exists}(M)$ in general does not determine
${\rm Th}_{\forall^2\exists}(M)$:
Let $\Gamma_2,\Gamma_3,\Gamma_4$ denote the following three tournaments
on 2, 3 respectively 4 vertices:
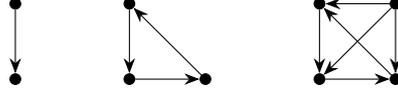
\tikzstyle{every node}=[circle, draw, fill=black!100,
                        inner sep=0pt, minimum width=4pt]
\begin{center}
\begin{tikzpicture}[scale=1.0]
    \node (1a) at (0,1) [] {}; 
    \node (2a) at (0,0) [] {}; 
    \draw [-{Stealth[length=2mm]}] { (1a) -- (2a) };
    \node (1b) at (1.5,1) [] {}; 
    \node (2b) at (1.5,0) [] {}; 
    \node (3b) at (2.5,0) [] {}; 
    \draw [-{Stealth[length=2mm]}] { (1b) -- (2b) };
    \draw [-{Stealth[length=2mm]}] { (2b) -- (3b) };
    \draw [-{Stealth[length=2mm]}] { (3b) -- (1b) };
    \node (1c) at (4.0,1.0) [] {}; 
    \node (2c) at (4.0,0.0) [] {}; 
    \node (3c) at (5.0,0.0) [] {}; 
    \node (4c) at (5.0,1.0) [] {}; 
    \draw [-{Stealth[length=2mm]}] { (1c) -- (2c) };
    \draw [-{Stealth[length=2mm]}] { (2c) -- (3c) };
    \draw [-{Stealth[length=2mm]}] { (3c) -- (1c) };
    \draw [-{Stealth[length=2mm]}] { (4c) -- (1c) };
    \draw [-{Stealth[length=2mm]}] { (4c) -- (2c) };
    \draw [-{Stealth[length=2mm]}] { (4c) -- (3c) };
\end{tikzpicture}
\end{center}
\noindent
Let $N=\omega\Gamma_2\cup\omega\Gamma_4$
and $M=N\cup\Gamma_3$, where $\cup$ denotes disjoint union
and
$\omega\Gamma_i$ the disjoint union of $\omega$ many copies of $\Gamma_i$.
It is easy to see that for every $a,b\in M$ there is an embedding $f_{ab}\colon N\rightarrow M$ with
$a,b\in f_{ab}(N)$,
and for every $c,d\in N$ there is an embedding $f_{cd}\colon M\rightarrow N$ with
$c,d\in f_{cd}(M)$
(as $\Gamma_3$ embeds into $\Gamma_4$).
So by $(f)\Rightarrow(a)$ in Lemma \ref{lem:AnE_abstract}, 
${\rm Th}_{\forall_2\exists}(M)={\rm Th}_{\forall_2\exists}(N)$.
However, the following sentence is in 
${\rm Th}_{\forall^2\exists}(N)\setminus {\rm Th}_{\forall^2\exists}(M)$, as it holds in $\Gamma_2$ and $\Gamma_4$, but not in $\Gamma_3$:
$$
 \forall x\Bigg(
 (\forall y\neg xEy) \vee
 (\forall y\neg yEx) \vee
 \exists z_1,\dots,z_3\bigg(\bigwedge_{i\neq j}\neg z_i=z_j \wedge \bigwedge_{i}(xEz_i \vee z_iEx)\bigg)\Bigg).
$$
\end{example}

\section{Coding using $p$-th powers}
\label{sec:coding}

\noindent
Every time a structure $M$ carries a surjection $M\rightarrow M^r$ definable by a formula in a fragment like $\exists$, this allows reducing a block of $r$ universal quantifiers to a single universal quantifier followed by some extra existential quantifiers.
We now work out this idea in detail for definable surjections involving $p$-th powers in fields of characteristic $p$.
For the notions of $p$-independence and $p$-bases, and their basic properties see e.g.~\cite[\S2.7]{FJ}.
The language of rings is $\mathfrak{L}_{\rm ring}=\{+,-,\cdot,0,1\}$.
For the propositions in this section,
the reader who skipped our definition of \new{(computable) functors of formulas} can safely replace the fragment $\mathsf{F}$ by the (computable) fragment~$\exists$.

We first exploit that a surjection $M\rightarrow M^r$ can be defined with parameters for a $p$-basis.
\new{This is well-known, but we provide a short proof for lack of a suitable reference:}

\begin{lemma}\label{lem:chi}
For every $p>0$, every $n\geq 1$ and every $r\geq 0$ there exists $\chi_{p,n,r}(x,y_1,\dots,y_r,z_1,\dots,z_n)\in\Form_\exists({\mathfrak{L}_{\rm ring}})$
such that
if 
$K$ is a field of characteristic $p$
with $[K:K^p]=p^n$ and
$a_1,\dots,a_n$ is a $p$-basis of $K$,
then $\chi_{p,n,r}(x,\underline{y},\underline{a})$ defines a surjection $K\rightarrow K^r$.
The map $(p,n,r)\mapsto \chi_{p,n,r}$ is computable.
\end{lemma}

\begin{proof}
Take $\chi_{p,n,0}=\top$ and 
$\chi_{p,n,1}=(x=y_1)$.
Take $\chi_{p,n,2}$ to be
$$
\exists \lambda_{0,\dots,0},\dots,\lambda_{p-1,\dots,p-1}\Bigg(x=\sum_{0\leq i_{1},\ldots,i_{n}<p}\lambda_{i_{1},\ldots,i_{n}}^{p}\prod_{j=1}^{n}z_j^{i_{j}}\wedge
y_1=\lambda_{0,\dots,0}\wedge
y_2=\lambda_{p-1,\dots,p-1}\Bigg)
$$
and inductively
$\chi_{p,n,r+1}(x,y_1,\dots,y_{r+1},\underline{z})=\exists w(\chi_{p,n,r}(x,y_1,\dots,y_{r-1},w,\underline{z})\wedge\chi_{p,n,2}(w,y_r,y_{r+1},\underline{z}))$.
\end{proof}

\begin{proposition}
\label{prop:coding_param}
Let $p>0$, let $n\geq 1$,
and let $\underline{c}$ be an $n$-tuple of new constant symbols.
Let $\mathsf{F}$ be a
fragment such that
$\exists\mathsf{F}=\mathsf{F}$
and
$\Form_\mathsf{F}(\mathfrak{L})$
includes all quantifier-free
$\mathfrak{L}_{\mathrm{ring}}(\underline{c})$-formulas,
for all $\mathfrak{L}\supseteq\mathfrak{L}_{\rm ring}(\underline{c})$.
For every such $\mathfrak{L}$
there exists a 
map
$$
 \tau_\mathfrak{L}\colon\Form_{\forall\mathsf{F}}(\mathfrak{L})\rightarrow\Form_{\forall_{1}[\mathsf{F}]}(\mathfrak{L})
$$
such that
for every field $K$ of characteristic $p$ with $[K:K^p]=p^{n}$,
for every $p$-basis $\underline{a}$ of $K$,
and
for every $\mathfrak{L}$-structure $K'$ expanding the $\mathfrak{L}_{\rm ring}(\underline{c})$-structure $(K,\underline{a})$,
we have
$K'\models(\varphi\leftrightarrow\tau_{\mathfrak{L}}\varphi)$
and ${\rm Var}(\varphi)={\rm Var}(\tau_{\mathfrak{L}}\varphi)$
for all $\varphi\in\Form_{\forall\mathsf{F}}(\mathfrak{L})$;
in particular, ${\rm Th}_{\forall\mathsf{F}}(K')=\tau_{\mathfrak{L}}^{-1}({\rm Th}_{\forall_1[\mathsf{F}]}(K'))$.
Moreover
$\tau_{\mathfrak{L}}$ is uniform in such $\mathfrak{L}$,
and if $\mathsf{F}$ is computable,
then $\tau_{\mathfrak{L}}$ is computable uniformly in $\mathfrak{L}$, $p$ and $n$.
\end{proposition}

\begin{proof}
Let $\mathfrak{L}\supseteq\mathfrak{L}_{\rm ring}(\underline{c})$
and let
$\varphi(\underline{u})\in\Form_{\forall\mathsf{F}}(\mathfrak{L})$.
Then $\varphi$ is equivalent to
${\rm prnx}_{\mathsf{F},\mathfrak{L}}(\varphi)=\forall x_{1},\ldots, x_{r}\psi(\underline{x},\underline{u})$
for some $r$,
with $\psi(\underline{x},\underline{u})\in\Form_{\mathsf{F}}(\mathfrak{L})$,
see Remark \ref{rem:prnx}.
Let $\tau_{\mathfrak{L}}\varphi$ be the formula
$$
 \forall w\exists\underline{x}(\chi_{p,n,r}(w,\underline{x},\underline{c})\wedge\psi(\underline{x},\underline{u})),
$$
with $\chi_{p,n,r}(w,\underline{x},\underline{c})\in\Form_{\exists}(\mathfrak{L}_{\mathrm{ring}}(\underline{c}))\subseteq\Form_{\mathsf{F}}(\mathfrak{L})$ as in Lemma~\ref{lem:chi}.
Then $\tau_{\mathfrak{L}}\varphi\in\Form_{\forall_{1}[\mathsf{F}]}(\mathfrak{L})$
is as required.
This is uniform in $\mathfrak{L}$, since ${\rm prnx}_{\mathsf{F},\mathfrak{L}}$ is.
If $\mathsf{F}$ is computable, then 
${\rm prnx}_{\mathsf{F},\mathfrak{L}}$
is uniformly computable (Remark \ref{rem:prnx}),
which together with the computability of $\chi_{p,n,r}$ gives the uniform computability of $\tau_\mathfrak{L}$.
\end{proof}

Next we quantify away the parameters for the $p$-basis.

\begin{lemma}\label{lem:pindep}
For every $p>0$ and every $n\geq1$ there exists  
$\pi_{p,n}(z_{1},\ldots,z_{n})\in\Form_\exists(\mathfrak{L}_{\rm ring})$
which defines in
every field $K$ of characteristic $p$
the set $D\subseteq K^n$ of 
$p$-dependent $n$-tuples.
The map $(p,n)\mapsto\pi_{p,n}$ is computable.
\end{lemma}

\begin{proof}
Take $\pi_{p,n}$ to be
$$
\exists \lambda_{0,\dots,0},\dots,\lambda_{p-1,\dots,p-1}\Bigg(\sum_{0\leq i_{1},\ldots,i_{n}<p}\lambda_{i_{1},\ldots,i_{n}}^{p}\prod_{j=1}^{n}z_j^{i_{j}}=0\wedge
\bigvee_{0\leq i_{1},\ldots,i_{n}<p}\neg \lambda_{i_{1},\ldots,i_{n}}=0\Bigg).
$$
\end{proof}

\begin{remark}
If $[K:K^p]=p^n$ then $B:=K^{n}\setminus D$ with $D$ as in Lemma \ref{lem:pindep} is the set of $p$-bases of $K$,
and $\chi_{p,n,r}$ defines
a surjection $B\times K\rightarrow K^r$
on the $\forall$-definable set $B\times K\subseteq K^{n+1}$.
This can be extended to a surjection
$K^{n+1}\rightarrow K^r$,
which however is then only 
definable by the conjunction of an existential and a universal formula.
\end{remark}

\begin{proposition} 
\label{prop:coding_noparam}
Let $p>0$ and $n\geq 1$.
Let $\mathsf{F}$ be a fragment such that
$\exists\mathsf{F}=\mathsf{F}$
and
$\Form_\mathsf{F}(\mathfrak{L})$ includes all quantifier-free $\mathfrak{L}_{\mathrm{ring}}$-formulas,
for all $\mathfrak{L}\supseteq\mathfrak{L}_{\rm ring}$.
For every such $\mathfrak{L}$,
there exists a map
$$
 \tau_\mathfrak{L}\colon\Form_{\forall\mathsf{F}}(\mathfrak{L})\rightarrow\Form_{\forall_{n+1}[\mathsf{F}]}(\mathfrak{L})
$$
such that
for every field $K$ of characteristic $p$ with $[K:K^p]=p^{n}$,
and for every $\mathfrak{L}$-structure $K'$ expanding $K$,
we have
$K'\models(\varphi\leftrightarrow\tau_{\mathfrak{L}}\varphi)$
and ${\rm Var}(\varphi)={\rm Var}(\tau_{\mathfrak{L}}\varphi)$
for all $\varphi\in\Form_{\forall\mathsf{F}}(\mathfrak{L})$;
in particular
${\rm Th}_{\forall\mathsf{F}}(K')=\tau_{\mathfrak{L}}^{-1}({\rm Th}_{\forall_{n+1}[\mathsf{F}]}(K'))$.
Moreover
$\tau_{\mathfrak{L}}$ is uniform in such $\mathfrak{L}$,
and if $\mathsf{F}$ is computable,
then $\tau_{\mathfrak{L}}$ is computable
uniformly in $\mathfrak{L}$, $p$ and $n$.
\end{proposition}

\begin{proof}
Let $\mathfrak{L}\supseteq\mathfrak{L}_{\rm ring}$ and 
let $\varphi(\underline{u})\in\Form_{\forall\mathsf{F}}(\mathfrak{L})$.
Again we may assume that $\varphi(\underline{u})$ is of the form
$\forall x_{1},\ldots, x_{r}\psi(\underline{x},\underline{u})$
for some $r$,
with $\psi(\underline{x},\underline{u})\in\Form_{\mathsf{F}}(\mathfrak{L})$ (Remark \ref{rem:prnx}).
Let $\tau_{\mathfrak{L}}\varphi$
be the formula
$$
 \forall w, z_{1},\ldots, z_{n}(\pi_{p,n}(\underline{z})\vee\exists \underline{x}(\chi_{p,n,r}(w,\underline{x},\underline{z})\wedge\psi(\underline{x},\underline{u}))),
$$
with
$\chi_{p,n,r},\pi_{p,n}\in\Form_{\exists}(\mathfrak{L}_{\mathrm{ring}})\subseteq\Form_{\mathsf{F}}(\mathfrak{L})$
as in Lemmas~\ref{lem:chi}
and~\ref{lem:pindep}.
Then
$\tau_{\mathfrak{L}}\varphi\in\Form_{\forall_{n+1}[\mathsf{F}]}(\mathfrak{L})$ is as required.
Uniformity and uniform computability follow as in the proof of Proposition \ref{prop:coding_param}.
\end{proof}

In the special case of function fields in which the constant field is definable, this can be improved further:

\begin{proposition} 
\label{prop:coding_functionfields}
Let $p>0$, let $n\geq 0$,
and let $d\geq 1$.
Let $\mathsf{F}\in\{\exists,\exists\forall,\exists\forall\exists,\dots\}$
and let $\gamma(x)\in\Form_{\exists}(\mathfrak{L}_{\mathrm{ring}})$.
There exists a map
$$
 \tau\colon\Form_{\forall\mathsf{F}}(\mathfrak{L}_{\mathrm{ring}})\rightarrow\Form_{\forall_{d+1}[\mathsf{F}]}(\mathfrak{L}_{\mathrm{ring}})
$$
such that for every 
geometrically integral $\mathbb{F}_p$-variety $X$ of dimension $d$
and every field $k$ with
$\mathrm{char}(k)=p$,
$[k:k^{p}]=p^{n}$,
and $\gamma(k(X))=k$,
we have
$k(X)\models((\varphi\wedge\bigwedge_{i=1}^{m}\gamma(u_i))\leftrightarrow\tau\varphi)$
and
$\mathrm{Var}(\varphi)=\mathrm{Var}(\tau\varphi)$
for all $\varphi(u_{1},\ldots,u_{m})\in\Form_{\forall\mathsf{F}}(\mathfrak{L}_{\mathrm{ring}})$;
in particular
${\rm Th}_{\forall\mathsf{F}}(k(X))=\tau^{-1}({\rm Th}_{\forall_{d+1}[\mathsf{F}]}(k(X)))$.
This $\tau$ is computable uniformly in $p$, $n$, $d$ and $\gamma$.
\end{proposition}

\begin{proof}
Let $k$ and $\varphi$ be as in the statement, and let $K=k(X)$.
Without loss of generality, $\varphi(\underline{u})$ is of the form
$\forall x_1,\dots, x_r\psi(\underline{x},\underline{u})$
with $\psi\in\Form_{\mathsf{F}}(\mathfrak{L}_{\rm ring})$ (Remark \ref{rem:prnx}).
For every $\nu\geq 1$,
$$
 K^{p^\nu}\;\cong\; K\;=\;\mathbb{F}_p(X)k\;\cong_k\;\mathbb{F}_p(X)^{p^\nu}k\;=\;(\mathbb{F}_p(X)k)^{p^\nu}k\;=\;K^{p^\nu}k
$$
and the degrees are as in the following diagram
(cf.~\cite[V.135 No.~6 Cor.~3]{Bourbaki_Algebra2}):
$$
 \xymatrix{
 \mathbb{F}_p(X)\ar@{-}[rr] & & K\\
 \mathbb{F}_p(X)^{p^\nu}\ar^{p^d}@{-}[u]\ar@{-}[r] & K^{p^\nu}\ar^{p^{n\nu}}@{-}[r]\ar^{p^{(n+d)\nu}}@{-}[ru]& K^{p^\nu}k\ar_{p^{d\nu}}@{-}[u] \\
 \mathbb{F}_p\ar@{-}[u]\ar@{-}[r]&k^{p^\nu}\ar@{-}[u]\ar^{p^{n\nu}}@{-}[r]&k\ar@{-}[u]
 }
$$
In particular, as $[K^{p^\nu}k:K^{p^\nu}]=p^{n\nu}$,
the formula
$$
 \eta_\nu(x) := 
 \exists y_1,\dots,y_{p^{n\nu}},z_1,\dots,z_{p^{n\nu}}\left(\bigwedge_{i=1}^{p^{n\nu}}\gamma(z_i) \wedge x = \sum_{i=1}^{p^{n\nu}}y_i^{p^\nu}z_i\right),
$$
defines $K^{p^\nu}k$ in $K$,
and the formula $\pi'(\underline{z})$ given by
$$
  \exists \lambda_{0,\dots,0},\dots,\lambda_{p-1,\dots,p-1}\bigg(
 \bigwedge_{0\leq i_{1},\ldots,i_{d}<p}\eta_1(\lambda_{i_1,\dots,i_d})\wedge
 \sum_{0\leq i_{1},\ldots,i_{d}<p}\lambda_{i_{1},\ldots,i_{d}}\prod_{j=1}^{d}z_j^{i_{j}}=0\wedge
\bigvee_{0\leq i_{1},\ldots,i_{d}<p}\neg \lambda_{i_{1},\ldots,i_{d}}=0\bigg)
$$
defines the set $d$-tuples that are not $p$-bases of $K$ over $K^{p}k$ (cf.~proof of Lemma \ref{lem:pindep}).
Note that $\eta_\nu,\pi'\in{\rm Form}_{\exists}(\mathfrak{L}_{\rm ring})$.
We write $\psi^{\eta_\nu}$ for $\psi$ with all quantifiers relativized to $\eta_\nu$, 
defined inductively by
$(\exists x\alpha)^{\eta_\nu}=\exists x(\eta_\nu(x)\wedge\alpha^{\eta_\nu})$
and 
$(\forall x\alpha)^{\eta_\nu}=\forall x(\eta_\nu'(x)\vee\alpha^{\eta_\nu})$,
where $\eta_\nu'={\rm prnx}(\neg\eta_\nu)\in\Form_{\forall}(\mathfrak{L}_{\rm ring})$, cf.~Remark \ref{rem:prnx}.
Note that $\psi^{\eta_\nu}\in{\rm Form}_\mathsf{F}(\mathfrak{L}_{\rm ring})$,
as $\exists=\exists\exists$, $\forall\exists=\forall\forall\exists$, etc.
Now fix $\nu$ with $p^{\nu-1}<r\leq p^{\nu}$ and let 
\begin{eqnarray*}
 \tau\varphi(\underline{u}) \;:=\; \forall w,z_1,\dots,z_d\exists \underline{\lambda} \Bigg(\pi'(\underline{z})&\vee&\bigg(w= \sum_{0\leq i_{1},\ldots,i_{d}<p^\nu}\lambda_{i_{1},\ldots,i_{d}}\prod_{j=1}^{d}z_j^{i_{j}}\\
 &&\wedge\bigwedge_{0\leq i_1,\dots,i_d<p^\nu}\eta_\nu(\lambda_{i_1,\dots,i_d})\wedge\bigwedge_{j=1}^m\gamma(u_j) 
 \wedge\psi^{\eta_\nu}(\lambda_{0,\dots,0},\dots,\lambda_{r-1,0,\dots,0},\underline{u})\bigg)\Bigg). 
\end{eqnarray*}
Then $\tau\varphi\in\Form_{\forall_{d+1}[\mathsf{F}]}(\mathfrak{L}_{\rm ring})$,
and
$K\models\tau\varphi(\underline{a})$ if and only if
$a_1,\dots,a_m\in k$
and
$K^{p^\nu}k\models\varphi(\underline{a})$,
which since $K^{p^\nu}k\cong_k K$ is equivalent to $K\models\varphi(\underline{a})$.
\end{proof}

\begin{example}
Let 
$K=\mathbb{F}_p(\!(t)\!)(s)$.
Each of the three propositions in this section gives a 
class of fields $\mathcal{K}$ with $K\in\mathcal{K}$, and a
computable map $\tau$ such that $K'\models(\varphi\leftrightarrow\tau\varphi)$ for every
$K'\in\mathcal{K}$ and every
$\varphi\in{\rm Sent}_{\forall\exists}(\mathfrak{L}_{\rm ring})$, where
\begin{enumerate}[$(a)$]
\item $\tau\varphi\in{\rm Sent}_{\forall_1[\exists]}(\mathfrak{L}_{\rm ring}(s,t))$ in Proposition \ref{prop:coding_param} (with $n=2$),
\item $\tau\varphi\in{\rm Sent}_{\forall_3[\exists]}(\mathfrak{L}_{\rm ring})$ in Proposition \ref{prop:coding_noparam} (with $n=2$), and
\item $\tau\varphi\in{\rm Sent}_{\forall_2[\exists]}(\mathfrak{L}_{\rm ring})$ in Proposition \ref{prop:coding_functionfields} (with $n=1,d=1$; for existence of $\gamma$ see Example \ref{ex:large} below).
\end{enumerate}
\end{example}

\section{Function fields}
\label{sec:ff}

\noindent
We consider the rational function field $k(t)$ in $\mathfrak{L}_{\rm ring}$ as well as in $\mathfrak{L}_{\rm ring}(t)$.
First we study to what extent the theory of $k$ determines the theory of $k(t)$,
and then we look for relations between the $\forall_1\exists$-theory of $k(t)$ and the $\exists$-theory of $(k(t),t)$.
Where possible we consider more generally the function field $k(X)$ of a variety $X$ defined over a subfield $k_0$ of $k$ as a structure in the language $\mathfrak{L}_{\rm ring}(k_0)$ or $\mathfrak{L}_{\rm ring}(k_0(X))$, which we then denote by
$(k(X),k_0)$, respectively $(k(X),k(X))$.
Throughout this section, $k,l,k_0,k'$ denote fields.

\begin{proposition}\label{prop:kt_Et}
If $k,l$ are extensions of $k_0$ with
${\rm Th}_\exists(k,k_0)\subseteq{\rm Th}_\exists(l,k_0)$,
then 
$$
 {\rm Th}_\exists(k(X),k_0(X))\subseteq{\rm Th}_\exists(l(X),k_0(X))
$$
for every geometrically integral $k_0$-variety $X$.
In particular, if $k,l$ are any fields with
${\rm Th}_\exists(k)\subseteq{\rm Th}_\exists(l)$,
then
${\rm Th}_\exists(k(t),t)\subseteq{\rm Th}_\exists(l(t),t)$.    
\end{proposition}

\begin{proof}
Let $\kappa=\max\{|k|,\aleph_0\}$
and let $l^\mathcal{U}$ be a $\kappa^+$-saturated ultrapower of $l$
(so $\mathcal{U}$ is an ultrafilter on some index set $I$, and $l^\mathcal{U}=l^I/\mathcal{U}$).
Then ${\rm Th}_\exists(k,k_0)\subseteq{\rm Th}_\exists(l,k_0)$
implies that there exists an $\mathfrak{L}_{\rm ring}(k_0)$-embedding $k\rightarrow l^\mathcal{U}$ \cite[Lemma 5.2.1]{CK},
which extends uniquely to an $\mathfrak{L}_{\rm ring}(k_0(X))$-embedding $k(X)\rightarrow l^\mathcal{U}(X)$.
The inclusion $l\rightarrow l(X)$ gives rise to an $\mathfrak{L}_{\rm ring}(l)$-embedding
$l^\mathcal{U}\rightarrow l(X)^\mathcal{U}$,
which extends uniquely to an $\mathfrak{L}_{\rm ring}(l(X))$-embedding
$l^\mathcal{U}(X)\rightarrow l(X)^\mathcal{U}$.
Therefore 
$$
 {\rm Th}_\exists(k(X),k_0(X))\subseteq{\rm Th}_\exists(l^\mathcal{U}(X),k_0(X))\subseteq{\rm Th}_\exists(l(X)^\mathcal{U},k_0(X))={\rm Th}_\exists(l(X),k_0(X)).
$$ 
The ``in particular'' part follows,
since if $k,l$ are fields with ${\rm Th}_\exists(k)\subseteq{\rm Th}_\exists(l)$, then in particular ${\rm char}(k)={\rm char}(l)$,
and so without loss of generality they have the same prime field $k_0$, and we can apply the claim to the $k_0$-variety $X=\mathbb{A}^1$.
\end{proof}

\begin{lemma}\label{lem:A1Ekt}
If $k\preccurlyeq_\exists k'$
then
 ${\rm Th}_{\forall_1\exists}(k(t))={\rm Th}_{\forall_1\exists}(k'(t))$.   
\end{lemma}

\begin{proof}
Firstly, $k\preccurlyeq_\exists k'$ implies that
$k(t)\preccurlyeq_\exists k'(t)$,
see \cite[Lemma 7.2]{DDF},
from which it follows that 
$\mathrm{Th}_{\forall\exists}(k'(t))\subseteq\mathrm{Th}_{\forall\exists}(k(t))$
(Lemma \ref{lem:AE_abstract}),
in particular 
$\mathrm{Th}_{\forall_1\exists}(k'(t))\subseteq\mathrm{Th}_{\forall_1\exists}(k(t))$.
For the converse inclusion,
we verify Lemma \ref{lem:AnE_abstract}$(f)$
for $N=k'(t)$, $M=k(t)$ and $n=1$.
So let $b\in k'(t)$.
If $b\in k$,
we extend the inclusion $\iota\colon k\rightarrow k'$
to $\iota'\colon k(t)\rightarrow k'(t)$ by $\iota'(t)=t$,
and note that $b$ is in the image.
If $b\notin k$ then,
as both $k'/k$ (by e.g.~\cite[Corollary 3.1.3]{Ershov}) and $k'(t)/k'$ are regular, $b$ is transcendental over $k$, 
so we may extend $\iota$ to an embedding
$\iota''\colon k(t)\rightarrow k'(t)$
with $\iota''(t)=b$.
So $(f)$ holds with $N_b=N$ and $M_b=M$.
By $(f)\Rightarrow(a)$ of Lemma~\ref{lem:AnE_abstract}, 
we have $\mathrm{Th}_{\forall_{1}\exists}(k(t))\subseteq\mathrm{Th}_{\forall_{1}\exists}(k'(t))$.
\end{proof}

\begin{proposition}\label{prop:kt_A1E}
 If ${\rm Th}_{\forall_1\exists}(k)\subseteq{\rm Th}_{\forall_1\exists}(l)$, then
 ${\rm Th}_{\forall_1\exists}(k(t))\subseteq{\rm Th}_{\forall_1\exists}(l(t))$.
\end{proposition}

\begin{proof}
From ${\rm Th}_{\forall_1\exists}(k)\subseteq{\rm Th}_{\forall_1\exists}(l)$
it follows by Lemma \ref{lem:AnE_abstract}$(d)$ that there exists $l^{*}\succcurlyeq l$
such that for every $b\in l^{*}$ there exist $k_{b}\succcurlyeq k$ and an embedding
$f_{b}\colon k_{b}\rightarrow l^{*}$ 
with $b\in f_b(k_b)$.
By Lemma \ref{lem:A1Ekt}
$\mathrm{Th}_{\forall_{1}\exists}(l^*(t))=\mathrm{Th}_{\forall_{1}\exists}(l(t))$.

To see that
${\rm Th}_{\forall_1\exists}(k(t))\subseteq{\rm Th}_{\forall_1\exists}(l^*(t))$,
we verify Lemma \ref{lem:AnE_abstract}$(f)$ for $M=k(t)$, $N=l^*(t)$ and $n=1$.
So let $b\in l^{*}(t)$.
If $b\in l^{*}$ then
we extend $f_{b}$
to an embedding
$f_{b}'\colon k_{b}(t)\rightarrow l^{*}(t)$
by $f_b'(t)=t$.
In this case we choose $M_{b}=k_{b}(t)$ and $N_{b}=N$.
If on the other hand
$b\notin l^{*}$
then it is transcendental over $l^{*}$,
in particular over $f_{1}(k_1)$.
We extend $f_{1}$ to an embedding
$f_{1}''\colon k_{1}(t)\rightarrow l^{*}(t)$
by $f_1''(t)=b$.
In this case we choose $M_{b}=k_{1}(t)$ and $N_{b}=N$.
Note that in both cases
$M_{b}\equiv_{\forall_{1}\exists}M$
by Lemma \ref{lem:A1Ekt}.
This verifies Lemma \ref{lem:AnE_abstract}$(f)$, 
and the result now follows from the implication 
$(f)\Rightarrow(a)$
of Lemma \ref{lem:AnE_abstract}.
\end{proof}

\begin{remark}
In particular, if $\mathbb{R}\preccurlyeq\mathbb{R}^*$ is a proper elementary extension,
then 
${\rm Th}_\exists(\mathbb{R}(t),t)={\rm Th}_\exists(\mathbb{R}^*(t),t)$ (Proposition \ref{prop:kt_Et}) and
${\rm Th}_{\forall_1\exists}(\mathbb{R}(t))={\rm Th}_{\forall_1\exists}(\mathbb{R}^*(t))$ (Proposition \ref{prop:kt_A1E}), but
one can show that 
${\rm Th}_{\forall_1\exists}(\mathbb{R}(t),t)\neq{\rm Th}_{\forall_1\exists}(\mathbb{R}^*(t),t)$,
${\rm Th}_{\forall^2\exists}(\mathbb{R}(t))\neq {\rm Th}_{\forall^2\exists}(\mathbb{R}^*(t))$
and
${\rm Th}_{\forall_3\exists}(\mathbb{R}(t))\neq{\rm Th}_{\forall_3\exists}(\mathbb{R}^*(t))$,
as will appear in \cite{Vollprecht}.
We do not know whether
${\rm Th}_{\forall_2\exists}(\mathbb{R}(t))={\rm Th}_{\forall_2\exists}(\mathbb{R}^*(t))$.

Suppose now that $k,l$ are perfect of characteristic $p>0$.
By Proposition \ref{prop:coding_param},
if ${\rm Th}_{\forall_1\exists}(k(t),t)={\rm Th}_{\forall_1\exists}(l(t),t)$,
then
${\rm Th}_{\forall\exists}(k(t),t)={\rm Th}_{\forall\exists}(l(t),t)$,
and if 
${\rm Th}_{\forall_1\exists}(k(t),t)$
is decidable, then so is
${\rm Th}_{\forall\exists}(k(t),t)$.
Similarly,
by Proposition \ref{prop:coding_noparam},
if ${\rm Th}_{\forall_2\exists}(k(t))={\rm Th}_{\forall_2\exists}(l(t))$,
then
${\rm Th}_{\forall\exists}(k(t))={\rm Th}_{\forall\exists}(l(t))$,
and decidability of
${\rm Th}_{\forall_2\exists}(k(t))$
implies decidability of
 ${\rm Th}_{\forall\exists}(k(t))$.
\end{remark}

\begin{lemma}\label{lem:EndAut}
For any field $k$,
$\{\sigma(t) : \sigma\in{\rm End}(k(t)/k)\} = k(t)\setminus k$.
\end{lemma}

\begin{proof}
Every $\sigma\in{\rm End}(k(t)/k)$ is injective,
so $\sigma^{-1}(k)=k$ and therefore $\sigma(t)\in k(t)\setminus k$.
Conversely, every $s\in k(t)/k$ is transcendental over $k$, so 
$\sigma|_k={\rm id}_k$, $\sigma(t)=s$
defines an isomorphism
$\sigma\colon k(t)\rightarrow k(s)$.
\end{proof}

\begin{proposition}\label{prop:rat_ff_Treda}
Let $q\in\mathbb{N}\cup\{\infty\}$.
There exists a map 
$$
 \tau\colon {\rm Form}_{\forall_1[\exists]}(\mathfrak{L}_{\rm ring})\rightarrow{\rm Form}_{\exists}(\mathfrak{L}_{\rm ring}(t))\times{\rm Form}_{\forall_1[\exists]}(\mathfrak{L}_{\rm ring})
$$
such that if $k$ is a field with $\#k=q$,
$\varphi\in{\rm Form}_{\forall_1[\exists]}(\mathfrak{L}_{\rm ring})$
and $\tau\varphi=(\psi_1,\psi_2)$, 
then 
${\rm Var}(\psi_1)={\rm Var}(\psi_2)={\rm Var}(\varphi)$ and
$$
 \varphi(k(t))\cap k^n = \psi_1((k(t),t))\cap\psi_2(k).
$$ 
In particular,
for $\varphi\in{\rm Sent}_{\forall_1[\exists]}(\mathfrak{L}_{\rm ring})$,
$$
 k(t)\models\varphi\quad\mbox{ if and only if }\quad
 (k(t),t)\models\psi_1\mbox{ and }k\models\psi_2.
$$ 
The map $\tau$ is computable, uniformly in $q$.
\end{proposition}

\begin{proof}
Let $\varphi=\forall x\psi$ with $\psi(x,\underline{u})\in{\rm Form}_\exists(\mathfrak{L}_{\rm ring})$.
If $q=\infty$,
define $\tau\varphi=(\psi(t,\underline{u}),\varphi)$, and if $q<\infty$ define
$\tau\varphi=(\varphi_q,\top(\underline{u}))$,
where 
$$
 \varphi_q := \psi(t,\underline{u})\wedge\exists y_1,\dots,y_q\Big(\bigwedge_i y_i^q-y_i=0\wedge\bigwedge_{i\neq j}y_i\neq y_j\wedge\bigwedge_i\psi(y_i,\underline{u})\Big).
$$
To prove the claim, let $k$ be a field with $q=\#k$, and let $\tau\varphi=(\psi_1,\psi_2)$.
First let $\underline{a}\in\varphi(k(t))\cap k^n$, i.e.~$\underline{a}\in k^n$
and for every $b\in k(t)$, $k(t)\models\psi(b,\underline{a})$.
Then in particular $k(t)\models\psi(t,\underline{a})$,
and if $q<\infty$, then also $k(t)\models\psi(b,\underline{a})$ for every $b\in k$, i.e.~for the zeros of $Y^q-Y$.
So $\underline{a}\in\psi_1((k(t),t))$ in all cases,
and if $q=\infty$, then $k\preccurlyeq_\exists k(t)$, see e.g.~\cite[Example 3.1.2]{Ershov}, and so $k(t)\models\varphi(\underline{a})$
implies $k\models\varphi(\underline{a})$ (Lemma \ref{lem:AE_abstract}),
and therefore $\underline{a}\in\psi_2(k)$ in all cases.

Conversely, let $\underline{a}\in\psi_1((k(t),t))\cap\psi_2(k)$.
Since $k(t)\models\psi(t,\underline{a})$,
it follows that $k(t)\models\psi(b,\underline{a})$
for every $b\in k(t)\setminus k$: Indeed,
by Lemma \ref{lem:EndAut} there exists a $\sigma_b\in{\rm End}(k(t)/k)$ with $\sigma_b(t)=b$,
hence $\sigma_b(k(t))\models\psi(b,\underline{a})$
and so, since $\sigma_b(k(t))\leq k(t)$ and $\psi$ is an existential formula,
$k(t)\models\psi(b,\underline{a})$.
If $q<\infty$, then $k(t)\models\psi_1(\underline{a})$ also implies that $k(t)\models\psi(b,\underline{a})$
for every $b\in k$,
and if $q=\infty$, then $k\models\psi_2(\underline{a})$ gives that
$k\models\psi(b,\underline{a})$ for every $b\in k$,
and so, since $k\leq k(t)$ and $\psi$ is an existential formula,
$k(t)\models\psi(b,\underline{a})$ for every $b\in k$.
This shows that in all cases $k(t)\models\psi(b,\underline{a})$ for every $b\in k(t)$,
i.e.~$\underline{a}\in\varphi(k(t))$.
\end{proof}

\begin{proposition}\label{prop:rat_ff_Tredb}
Let $\gamma(x)\in{\rm Form}_{\exists}(\mathfrak{L}_{\rm ring})$.
There exists a map
$$
 \tau\colon {\rm Form}_{\exists}(\mathfrak{L}_{\rm ring}(t))\rightarrow {\rm Form}_{\forall_1[\exists]}(\mathfrak{L}_{\rm ring})
$$
such that if 
$\varphi(\underline{u})\in{\rm Form}_{\exists}(\mathfrak{L}_{\rm ring}(t))$
and $k$ is a field with
$k\subseteq\gamma(k(t))\subseteq k(t)\setminus\{t\}$,
then ${\rm Var}(\tau\varphi)={\rm Var}(\varphi)$ and
$$
 \varphi((k(t),t)) \cap k^n = \tau\varphi(k(t))\cap k^n.
$$
In particular,  
${\rm Th}_{\exists}(k(t),t)=\tau^{-1}({\rm Th}_{\forall_1[\exists]}(k(t)))$.
The map $\tau$ is computable uniformly in $\gamma$.
\end{proposition}

\begin{proof}
Write $\varphi(\underline{u})=\psi(\underline{u},t)$ with $\psi\in{\rm Form}_\exists(\mathfrak{L}_{\rm ring})$ and define
$\tau\varphi:=\forall x(\psi(\underline{u},x)\vee\gamma(x))$.    
To see that this satisfies the claim,
first let $\underline{a}\in\varphi((k(t),t))\cap k^n$,
so $k(t)\models\psi(\underline{a},t)$.
As in the previous proof, by Lemma \ref{lem:EndAut} this implies that
$k(t)\models\psi(\underline{a},b)$ for every $b\in k(t)\setminus k$.
Since $k(t)\models\gamma(b)$ for every $b\in k$, it follows that $\underline{a}\in\tau\varphi(k(t))$.
Conversely, if $\underline{a}\in\tau\varphi(k(t))\cap k^n$, then
in particular $k(t)\models\psi(\underline{a},t)\vee\gamma(t)$, so $t\notin\gamma(k(t))$ shows that
$\underline{a}\in\varphi((k(t),t))$.
\end{proof}

\begin{example}\label{ex:large}
The condition that 
there exists $\gamma(x)\in{\rm Form}_\exists(\mathfrak{L}_{\rm ring})$
with
$k\subseteq \gamma(k(t))\subseteq k(t)\setminus\{t\}$
holds in particular
whenever $k$ is perfect of characteristic $p>0$,
as then we can take 
$\gamma(x)=\exists y(x=y^p)$,
which defines $k(t)^p$.
It also holds whenever
$k$ is existentially $\emptyset$-definable in $k(t)$,
which is trivially satisfied for $k$ finite,
and by \cite[Theorem 2]{Koenigsmann2002} also for any $k$ that is {\em large} in the sense that $k\preccurlyeq_\exists k(\!(t)\!)$.
\end{example}

\begin{corollary}\label{cor:kt_AE1_perfectlarge}
Let $k$ be a field of positive characteristic,
not containing the algebraic closure of the prime field.
If $k$ is perfect or large,
then ${\rm Th}_{\forall_1\exists}(k(t))$ is undecidable.
\end{corollary}

\begin{proof}
It is known that ${\rm Th}_\exists(k(t),t)$ is undecidable
whenever $k$ is of positive characteristic not containing the algebraic closure of the prime field.
For example this is implicit in \cite{ES17}.
By Proposition \ref{prop:rat_ff_Tredb}
and Example~\ref{ex:large},
if $k$ is in addition perfect or large, it follows that also ${\rm Th}_{\forall_1\exists}(k(t))$ is undecidable.
\end{proof}

\begin{corollary}
For $k$ algebraically closed, real closed, $p$-adically closed or finite,
${\rm Th}_\exists(k(t),t)$
is many-one equivalent to
${\rm Th}_{\forall_1\exists}(k(t))$.
\end{corollary}

\begin{proof}
For these $k$, $k$ is existentially $\emptyset$-definable in $k(t)$
by Example \ref{ex:large},
and ${\rm Th}(k)$ is decidable (in particular so is ${\rm Th}_{\forall_1\exists}(k)$),
so the claim follows from
Proposition \ref{prop:rat_ff_Treda}
and
Proposition \ref{prop:rat_ff_Tredb}.
\end{proof}

\begin{remark}
For $k$ algebraically closed of characteristic zero, the undecidability of ${\rm Th}_\exists(k(t),t)$ is a big open problem -
in fact we do not even know whether 
${\rm Th}(k(t),t)$ is undecidable,
see e.g.~\cite{Koenigsmann_survey}.
For $k$ finite, real closed
or $p$-adically closed,
the undecidability of 
${\rm Th}_\exists(k(t),t)$
is known by \cite{Pheidas,Videla}, \cite{Denef},
respectively \cite{MB,DD,BDD},
but even then the many-one equivalence gives additional computability theoretic information.
\end{remark}

\begin{proposition}\label{prop:Tred_ff}
Let $\mathsf{F}\in\{\exists,\exists\forall,\exists\forall\exists,\dots\}$.
   Let
   $k_0$ be a field of characteristic $0$
   and let
   $f\in k_0[\mathbf{x},\mathbf{y}]$ be absolutely irreducible
   such that 
   the curve $C:f=0$
   has function field $k_0(C)={\rm Frac}(k_0[\mathbf{x},\mathbf{y}]/(f))$ of genus at least $2$.
   Let $\gamma(z)\in{\rm Form}_\mathsf{F}(\mathfrak{L}_{\rm ring}(k_0))$.
   There is a map
   $$
    \tau\colon{\rm Form}_\mathsf{F}(\mathfrak{L}_{\rm ring}(k_0[\mathbf{x},\mathbf{y}]))\rightarrow{\rm Form}_{\forall_2[\mathsf{F}]}(\mathfrak{L}_{\rm ring}(k_0))
   $$
   such that if $\varphi(\underline{u})\in{\rm Form}_\mathsf{F}(\mathfrak{L}_{\rm ring}(k_0[\mathbf{x},\mathbf{y}]))$,
   then for every extension $k/k_0$ with
   $k\subseteq\gamma(k(C))\subseteq k(C)\setminus\{x\}$, we have
   ${\rm Var}(\varphi)={\rm Var}(\tau\varphi)$ and
   $$
    \varphi(k(C)) \cap k^n = \tau\varphi(k(C))\cap k^n,
   $$
   where $x,y$ are the residues of $\mathbf{x},\mathbf{y}$ in ${\rm Frac}(k[\mathbf{x},\mathbf{y}]/(f))$,
   and we interpret constants from $k_0[\mathbf{x},\mathbf{y}]$ as their residues in $k_0[x,y]$.
   In particular, 
   for $\varphi\in{\rm Sent}_\mathsf{F}(\mathfrak{L}_{\rm ring}(k_0[\mathbf{x},\mathbf{y}]))$, we get
   $k(C)\models(\varphi \leftrightarrow\tau\varphi)$.
   When given an injection $\alpha\colon k_0\rightarrow\mathbb{N}$,
   $\tau$ is computable uniformly in $f$ and $\gamma$.
\end{proposition}

\begin{proof}
Assume that $k/k_0$ is an extension with 
$k\subseteq\gamma(k(C))\subseteq k(C)\setminus\{x\}$.
Write $k(C)=k(x,y)$ and note that $f(x,y)=0$ and $x\notin k$. 
If $a,b\in k(C)\setminus k$ with $f(a,b)=0$, then $k(a,b)\cong_k k(C)$,
and since $k(x,y)/k(a,b)$ is separable and $k(x,y)$ and $k(a,b)$ are function fields of the same genus greater one, the Riemann--Hurwitz formula \cite[Theorem 3.4.13]{Stichtenoth} implies that $k(a,b)=k(x,y)$. In particular,
there is a unique $\sigma\in{\rm Aut}(k(C)/k)$ with $\sigma(x)=a$, $\sigma(y)=b$.
Conversely, if $\sigma\in{\rm Aut}(k(C)/k)$, then $f(\sigma(x),\sigma(y))=0$ and $\sigma(x)\notin k$.

Given $\varphi(\underline{u})\in\Form_\mathsf{F}(\mathfrak{L}_{\rm ring}(k_0[\mathbf{x},\mathbf{y}]))$, 
express each constant $c$ occurring in $\varphi$  
as $t(\mathbf{x},\mathbf{y})$ for an $\mathfrak{L}_{\rm ring}(k_0)$-term $t(v,w)$,
and this way obtain 
$\varphi_0(\underline{u},v,w)\in{\rm Form}_\mathsf{F}(\mathfrak{L}_{\rm ring}(k_0))$ such that
$k(C)\models(\varphi(\underline{u})\leftrightarrow\varphi_0(\underline{u},x,y))$,
independent of $k$.
Define
$$
 \tau\varphi := \forall a,b\big(f(a,b)\neq0\vee\gamma(a)\vee\varphi_0(\underline{u},a,b)\big).
$$
Then $\tau\varphi\in{\rm Form}_{\forall_2[\mathsf{F}]}(\mathfrak{L}_{\rm ring}(k_0))$ satisfies the claim:
Indeed, if $\underline{c}\in\varphi(k(C))\cap k^n$, then $k(C)\models\varphi_0(\underline{c},x,y)$,
and if $k(C)\models f(a,b)=0\wedge\neg\gamma(a)$, then, as explained above, there exists $\sigma\in{\rm Aut}(k(C)/k)$ with $\sigma(x)=a$, $\sigma(y)=b$, and therefore
$\varphi_0(\sigma(\underline{c}),\sigma(x),\sigma(y)) = \varphi_0(\underline{c},a,b)$ holds in $k(C)$,
so $\underline{c}\in\tau\varphi(k(C))$.
Conversely, if $\underline{c}\in\tau\varphi(k(C))\cap k^n$, then in particular $k(C)\models\varphi_0(\underline{c},x,y)$, and hence $\underline{c}\in\varphi(k(C))$.
\end{proof}

In positive characteristic we face additional difficulties, 
but we can use ideas from Section \ref{sec:coding} to reduce the two universal quantifiers to one, 
in a less straightforward way, 
\new{as the following result shows. Here, we use the terms
{\em presented} and {\em splitting algorithm} in the sense of \cite[Definitions 19.1.1, 19.1.2]{FJ}.}

\begin{proposition}\label{prop:Tred_ff_2}
   Let $k_0$ be a perfect field of characteristic $p>0$ and let 
   $f\in k_0[\mathbf{x},\mathbf{y}]$ be separable in $\mathbf{y}$ and absolutely irreducible
   such that
   the curve $C:f=0$
   has function field $k_0(C)={\rm Frac}(k_0[\mathbf{x},\mathbf{y}]/(f))$ of genus at least $2$.
   There is a map
   $$
    \tau\colon{\rm Form}_\exists(\mathfrak{L}_{\rm ring}(k_0[\mathbf{x},\mathbf{y}]))\rightarrow{\rm Form}_{\forall_1[\exists]}(\mathfrak{L}_{\rm ring}(k_0))
   $$
   such that if 
   $\varphi(\underline{u})\in{\rm Form}_\exists(\mathfrak{L}_{\rm ring}(k_0[\mathbf{x},\mathbf{y}]))$,
   then for every regular extension $k/k_0$ with $k$  perfect, 
   ${\rm Var}(\varphi)={\rm Var}(\tau\varphi)$ and
   $$
    \varphi(k(C)) \cap k^n = \tau\varphi(k(C))\cap k^n,
   $$
   where $x,y$ are the residues of $\mathbf{x},\mathbf{y}$ in ${\rm Frac}(k[\mathbf{x},\mathbf{y}]/(f))$
   and we interpret constants from $k_0[\mathbf{x},\mathbf{y}]$ by their residues in $k_0[x,y]$.
   In particular, for $\varphi\in{\rm Sent}_\exists(\mathfrak{L}_{\rm ring}(k_0[\mathbf{x},\mathbf{y}]))$, we get
   $k(C)\models(\varphi \leftrightarrow\tau\varphi)$.
   When given an injection $\alpha\colon k_0\rightarrow\mathbb{N}$,
   the map $\tau$ is computable,
   and if $k_0$ is moreover presented and has a splitting algorithm,
   then $\tau$ is computable uniformly in $f$.
\end{proposition}

\begin{proof}
    As $k_0$ is perfect
    and $k_0(C)$ is of genus at least two,
    $m:=|{\rm Aut}(k_0(C)/k_0)|<\infty$,
    see \cite{Rosenlicht}.
    Since $k_0$ is perfect, also the genus of $\overline{k_0}(C)$ is at least two, and so also ${\rm Aut}(\overline{k_0}(C)/\overline{k_0})$ is finite.
    It follows that
    ${\rm Aut}(\overline{k}(C)/\overline{k})={\rm Aut}(\overline{k_0}(C)/\overline{k_0})$,
    and therefore also
    ${\rm Aut}(k(C)/k)={\rm Aut}(k_0(C)/k_0)$ for every regular $k/k_0$.
    
    Let $(a,b)\in k(C)^2\setminus k^2$ with $f(a,b)=0$.
    Then $k(C)/k(a,b)$ is finite, 
    and $k(a,b)\cong_k k(C)$,
    so there exists $\sigma\in{\rm End}(k(C)/k)$ with $\sigma(x,y)=(a,b)$.
    If $k(C)/k(a,b)$ is separable, then in fact $k(C)=k(a,b)$ by the Riemann--Hurwitz formula,
    so $\sigma\in{\rm Aut}(k(C)/k)$.
    If on the other hand $k(C)/k(a,b)$ is inseparable, 
    then, as $[k(C):k(C)^p]=p$ due to the assumption that $k$ is perfect,
    $k(a,b)\subseteq k(C)^p$, and so in particular $a\in k(C)^p$.
    
   Now let $\varphi(\underline{u})\in {\rm Form}_\exists(\mathfrak{L}_{\rm ring}(k_0[\mathbf{x},\mathbf{y}]))$, and again rewrite constants by $\mathfrak{L}_{\rm ring}(k_0)$-terms 
   to obtain
   $\varphi_0(\underline{u},v,w)\in{\rm Form}_\exists(\mathfrak{L}_{\rm ring}(k_0))$ with
   $k(C)\models(\varphi(\underline{u})\leftrightarrow\varphi_0(\underline{u},x,y))$.
   As in the previous proof $\varphi_{0}(\underline{u},v,w)$ is independent of $k$.
   Define
\begin{eqnarray*}
\tau\varphi&:=&\forall z \left( \exists w(w^p=z) \vee\eta(\underline{u},z) \right)
\end{eqnarray*}
   with 
\begin{eqnarray*}
\eta(\underline{u},z) &:=&
\exists x_1,\dots,x_m,y_1,\dots,y_m \bigg(\bigwedge_{i\neq j}(x_i\neq x_j\vee y_i\neq y_j)\wedge \bigwedge_{i=1}^mf(x_i,y_i)=0 \wedge \bigwedge_{i=1}^m\psi(\underline{u},z,x_i) \bigg)
\end{eqnarray*}
and 
    \begin{eqnarray*}
     \psi(\underline{u},z,\xi) &:=& \exists \lambda_0,\dots,\lambda_{p-1} \bigg(z=\sum_{j=0}^{p-1}\lambda_j^p {\xi^j} \wedge \big(f(\lambda_0,\lambda_1)\neq 0  
    \vee \exists w(w^p=\lambda_0)\vee\varphi_0(\underline{u},\lambda_0,\lambda_1)\big) \bigg) . 
   \end{eqnarray*}
Then $\tau\varphi\in{\rm Form}_{\forall_1[\exists]}(\mathfrak{L}_{\rm ring}(k_0))$ satisfies the claim:

   Indeed, if 
   $\underline{d}\in\tau\varphi(k(C))\cap k^n$
  then in particular 
  $k(C)\models\eta(\underline{d},c_0)$
  for $c_0:=x^p+xy^p$,
  as the separability of $f$ in $\mathbf{y}$ implies that
   $x\notin k(C)^p$ and therefore $c_0\notin k(C)^p$.
   Thus there exist pairwise distinct $(x_1,y_1),\dots,(x_m,y_m)$ with  $f(x_i,y_i)=0$ 
   such that $\psi(\underline{d},c_0,x_i)$ holds for every $i$.
   In particular, $\sum_{j=0}^{p-1} \lambda_j^px_i^j=c_0\notin k(C)^p$,
   so in particular $x_i\notin k(C)^p$, especially $x_i\notin k$.
   Therefore $(x_i,y_i)=\sigma(x,y)$ for some $\sigma\in{\rm Aut}(k(C)/k)$, so since $m=|{\rm Aut}(k(C)/k)|$ there exists $i_0$ with $(x_{i_0},y_{i_0})=(x,y)$.
   From $\psi(\underline{d},c_0,x_{i_0})$ we get 
   $\lambda_0,\dots,\lambda_{p-1}$ with 
   $$
    \sum_{j=0}^{p-1} \lambda_j^px^j=c_0=x^px^0+y^px^1,
   $$ 
   which since $x$ is a $p$-basis for $k(C)$ implies that $\lambda_0=x$ and $\lambda_1=y$,
   so since $f(\lambda_0,\lambda_1)=f(x,y)=0$ and $\lambda_0=x\notin k(C)^p$ we conclude that $\varphi_0(\underline{d},x,y)$ holds,
   i.e.~$\underline{d}\in\varphi(k(C))$.

   Conversely, suppose that 
   $\underline{d}\in\varphi(k(C))\cap k^n$,
   in particular
   $k(C)\models\varphi_0(\underline{d},x,y)$.
   Then also $\varphi_0(\underline{d},a,b)$ holds in $k(C)$ for all $(a,b)\in k(C)^2\setminus k^2$ with $f(a,b)=0$,
   as it holds in $k(a,b)$ by the isomorphism above,
   and therefore it holds in $k(C)\supseteq k(a,b)$ due to the fact that $\varphi_0$ is existential.
   Now take any $c\in k(C)\setminus k(C)^p$,
   let ${\rm Aut}(k(C)/k)=\{\sigma_1,\dots,\sigma_m\}$
   and define $(x_i,y_i)=\sigma_i(x,y)$.
   Then $(x_1,y_1),\dots,(x_m,y_m)$ are pairwise distinct
   as $x,y$ generate $k(C)$ over $k$, and
   $f(x_i,y_i)=0$ for each $i$.
   We claim that $k(C)\models\psi(\underline{d},c,x_i)$ for every $i$.
   Indeed, write $c=\sum_{j=0}^{p-1} \lambda_j^p x_i^j$, which is possible since
   $x_i\notin k(C)^p$, so $x_i$ is a $p$-basis of $k(C)$.
   Whenever $f(\lambda_0,\lambda_1)= 0$  and $\lambda_0\notin k(C)^p$, in particular $\lambda_0\notin k$,
   so $\varphi_0(\underline{d},\lambda_0,\lambda_1)$ holds.

If $k_0$ is presented and has a splitting algorithm,
then $\tau$ is computable uniformly in $f$,
   since
   ${\rm Aut}(k_0(C)/k_0)$ and in particular $m$ is computable,
   see \cite{Hess}.
\end{proof}

{\color{blue}
\begin{proof}[Simplified proof of Proposition \ref{prop:Tred_ff_2}]
Let $\varphi(\underline{u})\in {\rm Form}_\exists(\mathfrak{L}_{\rm ring}(k_0[\mathbf{x},\mathbf{y}]))$, and again rewrite constants by $\mathfrak{L}_{\rm ring}(k_0)$-terms to obtain
$\varphi_0(\underline{u},v,w)\in{\rm Form}_\exists(\mathfrak{L}_{\rm ring}(k_0))$ with
$k(C)\models(\varphi(\underline{u})\leftrightarrow\varphi_0(\underline{u},x,y))$.
As before, $\varphi_{0}(\underline{u},v,w)$ is independent of $k$.
Define
\begin{eqnarray*}
\tau\varphi &:=& \forall \zeta\exists \xi,\eta,\lambda_0,\dots,\lambda_{p-1}\bigg(f(\xi,\eta)=0\wedge \zeta=\sum_{j=0}^{p-1}\lambda_j^p\xi^j\wedge\varphi_0(\underline{u},\xi,\eta)\bigg)
\end{eqnarray*} 
Then $\tau\varphi\in{\rm Form}_{\forall_1[\exists]}(\mathfrak{L}_{\rm ring}(k_0))$ satisfies the claim:

Indeed, if $\underline{d}\in\tau\varphi(k(C))\cap k^n$, then in particular
there exist $x_0,y_0,c_0,\dots,c_{p-1}\in k(C)$ with $f(x_0,y_0)=0$, $x=\sum_{j=0}^{p-1}c_j^px_0^j$ and $k(C)\models\varphi_0(\underline{d},x_0,y_0)$.
Since $x\notin k(C)^p$ due to the assumption that $f$ is separable in $\mathbf{y}$, this implies $x_0\notin k(C)^p$.
In particular, $x_0\notin k^p=k$, hence $k(x_0,y_0)\cong k(C)$; more precisely, there is a $k$-isomorphism $\sigma\colon k(x_0,y_0)\rightarrow k(C)$ with $\sigma(x_0)=x$, $\sigma(y_0)=y$.
As $[k(C):k(C)^p]=p$ due to the assumption that $k$ is perfect,
$x_0\notin k(C)^p$ also implies that $k(C)/k(x_0,y_0)$ is separable, 
so since $C$ is of genus at least $2$, 
the Riemann--Hurwitz formula shows that $k(C)=k(x_0,y_0)$,
hence $\sigma\in{\rm Aut}(k(C)/k)$.
Therefore, $k(C)\models\varphi_0(\underline{d},x_0,y_0)$ implies 
$k(C)\models\varphi_0(\underline{d},x,y)$, hence
$\underline{d}\in\varphi(k(C))$.

Conversely, suppose that $\underline{d}\in\varphi(k(C))\cap k^n$.
We have that $f(x,y)=0$ and $k(C)\models\varphi_0(\underline{d},x,y)$.
Moreover, since $x\notin k(C)^p$ and $[k(C):k(C)^p]=p$,
for every $z\in k(C)$ there exist $c_0,\dots,c_{p-1}\in k(C)$ with $z=\sum_{j=0}^{p-1}c_j^px^j$.
Thus $\underline{d}\in \tau\varphi(k(C))$.

Note that $\tau$ is computable uniformly in $f$,
without a need for $k_0$ to be presented with a splitting algorithm.
\end{proof}
}

\begin{corollary}\label{cor:ff_A1E_gg1}
Let $k$ be a presented perfect field of positive characteristic not containing the algebraic closure of a finite field, and let $C$ be a smooth projective curve over $k$
of genus greater than $1$.
Then ${\rm Th}_{\forall_1\exists}(k(C),k)$ 
is undecidable.
\end{corollary}

\begin{proof}
With $k$ also $k(C)$ is presented, and 
\cite[Theorem 1.1]{ES17} gives that
${\rm Th}_{\exists}(k(C),k(C))$ is undecidable.
By choosing a plane affine curve $C':f(X,Y)=0$ birationally equivalent to $C$,
Proposition \ref{prop:Tred_ff_2} then immediately implies that
${\rm Th}_{\forall_1\exists}(k(C),k)$ is undecidable.
\end{proof}

\begin{remark}
The recent thesis of Tyrrell
contains a similar statement (\cite[Corollary 2.5.5]{Tyrrell_thesis}) that also includes the case of genus $1$,
but only for $k$ algebraic over the prime field.
\end{remark}

\begin{proof}[Proof of Theorem \ref{thm:intro_functionfield}]
Part $(a)$ is the ``in particular'' part of Proposition \ref{prop:kt_Et}.    
Part $(b)$ is Proposition \ref{prop:kt_A1E}.
Part $(c)$ follows from Proposition \ref{prop:rat_ff_Treda}.
Part $(d)$ follows from Proposition
\ref{prop:rat_ff_Tredb}.
Part $(e)$ follows from the case $n=1$
of 
Proposition \ref{prop:coding_param} and
Proposition \ref{prop:coding_noparam}.
\end{proof}

\section{Laurent series fields}
\label{sec:Laurent}

\noindent
In this final section
we first develop general results for various classes of henselian valued fields,
and then we discuss in more detail the special case of Laurent series fields $k(\!(t)\!)$.

We follow notation and conventions of \cite{EP}.
As usual we denote by $v_t$ the $t$-adic valuation on the rational function field $k(t)$,
by $k(t)^h$ the henselization of $k(t)$ with respect to $v_t$, and then $k(\!(t)\!)$ is the completion of $k(t)$ with respect to $v_t$.

We work in the usual three-sorted language of valued fields
$\mathfrak{L}_{\rm val}$, see for example \cite[Definition 3.4]{AF23}.
In such a many-sorted setting, we will understand the quantifiers in the fragments defined in Section \ref{sec:fragments} to run over any of the sorts,
but we introduce the new fragments $\forall^{\mathbf{k}}\exists$ and
$\forall^{\mathbf{k}}_n\exists$
which are defined (on languages containing $\mathfrak{L}_{\rm val}$) precisely like the fragments
$\forall\exists$ 
respectively $\forall_n\exists$ 
except that all universal quantifiers run over the residue field sort $\mathbf{k}$ only, cf.~\cite[Section 3]{AF16}.
For a valued field $(K,v)$ we denote by
$K$ the $\mathfrak{L}_{\rm ring}$-structure $K$, by $(K,v)$ the corresponding $\mathfrak{L}_{\rm val}$-structure, and 
by $(K,v,\pi)$, for $\pi\in K$,
the $\mathfrak{L}_{\rm val}(\varpi)$-structure $(K,v)$ with $\varpi$ interpreted as $\pi$.

Let $\bHen$ be the $\mathfrak{L}_{\mathrm{val}}$-theory of equicharacteristic henselian nontrivially valued fields,
$\bHepi$ the $\mathfrak{L}_{\mathrm{val}}(\varpi)$-theory 
of models of $\bHen$ in which the interpretation of $\varpi$ is a uniformizer
(i.e.~an element of smallest positive value),
$\bHeZ$ the $\mathfrak{L}_{\mathrm{val}}$-theory
of models of $\bHen$ 
with value group a $\mathbb{Z}$-group,
$\bHepiZ$ the $\mathfrak{L}_{\mathrm{val}}(\varpi)$-theory
of models of $\bHepi$ 
with value group a $\mathbb{Z}$-group,
and let $\bHed=\bHepi\cap{\rm Sent}(\mathfrak{L}_{\rm val})$.
Denote by $\bHen_0$ the theory of the models of $\bHen$ of characteristic zero,
similar for the other theories defined.
Note that $(k(\!(t)\!),v_t)\models\bHen\cup\bHeZ\cup\bHed$ and
$(k(\!(t)\!),v_t,t)\models\bHepi\cup\bHepiZ$.

\begin{remark}\label{R4}
\new{Several of the following results will be conditional on a 
hypothesis called \Rfour\ introduced in
\cite[\S2]{ADF23}, see also
\cite[Remark 3.18]{AF23}. 
We would like to point out that 
in the present paper,
\Rfour\ is never applied directly, but only through results in other works.}
\new{This hypothesis has several equivalent formulations, one of which is the following:}
\begin{itemize}
\item[\Rfour]
Every large field $k$ is existentially closed in every extension $K/k$ for which there exists a valuation $v$ on $K/k$ with residue field $Kv=k$.
\end{itemize}
It is a consequence of local uniformization, and in particular a consequence of resolution of singularities in positive characteristic,
see \cite[Proposition 2.3]{ADF23}.
\end{remark}

We start by quoting known reduction results for complete theories and for existential theories and then prove various similar results for universal-existential theories.

\begin{proposition}\label{prop:Ftt}
\begin{enumerate}[$(a)$]
    \item  There exists a 
    computable map $\tau\colon\mathrm{Sent}_{\exists}(\mathfrak{L}_{\rm val})\rightarrow\mathrm{Sent}_{\exists}(\mathfrak{L}_{\rm ring})$
    with   $\tau(\mathrm{Sent}_{\exists_n}(\mathfrak{L}_{\rm val}))\subseteq \mathrm{Sent}_{\exists_n}(\mathfrak{L}_{\rm ring})$
    for every $n$
    such that
    ${\rm Th}_\exists(K,v)=\tau^{-1}({\rm Th}_\exists(Kv))$
    for every $(K,v)\models\bHen$. 
    \item If \Rfour\ holds,
    there exists a 
    computable map $\tau\colon\mathrm{Sent}_{\exists}(\mathfrak{L}_{\rm val}(\varpi))\rightarrow\mathrm{Sent}_{\exists}(\mathfrak{L}_{\rm ring})$
    such that
    ${\rm Th}_\exists(K,v,\pi)=\tau^{-1}({\rm Th}_\exists(Kv))$
    for every $(K,v,\pi)\models\bHepi$. 
    \item 
     There exists a 
    computable map $\tau\colon\mathrm{Sent}(\mathfrak{L}_{\rm val})\rightarrow\mathrm{Sent}(\mathfrak{L}_{\rm ring})$
    such that
    ${\rm Th}(K,v)=\tau^{-1}({\rm Th}(Kv))$
    for every $(K,v)\models\bHeZ_0$. 
\end{enumerate}

\end{proposition}

\begin{proof}
Part $(a)$ is \cite[Proposition 3.26]{AF23}.
Part $(b)$ is \cite[Corollary 3.20(c)]{AF23}.
Part $(c)$ is an immediate consequence of the classical
Ax--Kochen/Ershov principle for equicharacteristic zero henselian valued fields;
the map $\tau$ can be obtained for example from
\cite[Corollary 2.23]{AF23}.
\end{proof}

\begin{lemma}\label{lem:AE_1}
If an embedding $\iota\colon(K,v)\rightarrow(L,w)$ 
of valued fields 
induces an isomorphism of residue fields $Kv\rightarrow Lw$,
then
$\mathrm{Th}_{\forall^{\mathbf{k}}\exists}(K,v,C)\subseteq\mathrm{Th}_{\forall^{\mathbf{k}}\exists}(L,w,\iota C)$
with constant symbols in the field sort
from $C\subseteq K$.
\end{lemma}

\begin{proof}
This follows from the fact that the satisfaction of
existential sentences (in this case with parameters from the residue field sort and from $C$) is preserved under embeddings.
\end{proof}

\begin{proposition}\label{lem:AE_3}\label{prof:AE_char0}
Let $(K,v,\pi_{K})\models\bHepi$
and let $(L,w,\pi_L)$ be an $\mathfrak{L}_{\mathrm{val}}(\varpi)$-structure with
$(L,w)\models\bHen$ and $0<w(\pi_L)<\infty$.
Suppose that $\mathrm{Th}_{\forall\exists}(Kv)\subseteq\mathrm{Th}_{\forall\exists}(Lw)$.
\begin{enumerate}[$(a)$]
\item If \Rfour\ holds, then $\mathrm{Th}_{\forall^{\mathbf{k}}\exists}(K,v,\pi_{K})\subseteq\mathrm{Th}_{\forall^{\mathbf{k}}\exists}(L,w,\pi_{L})$.
\item If $(L,w,\pi_L)\models\bHepiZ_0$,
 then ${\rm Th}_{\forall\exists}(K,v,\pi_K)\subseteq{\rm Th}_{\forall\exists}(L,w,\pi_L)$.
\end{enumerate}
\end{proposition}

\begin{proof}
We treat $(a)$ and $(b)$ simultaneously,
where we use that the assumption of $(b)$ implies that $Kv$ and $Lw$ are of characteristic zero so that \Rfour\ holds when restricted to extensions of these (see \cite[Remarks 2.4, 4.18]{ADF23}).
By \cite[Proposition 4.5]{ADF23},
we can replace $(K,v)$ by an elementary extension
to assume without loss of generality that
there is a section
$\zeta_{K}\colon Kv\rightarrow K$
of the residue map of $v$.
Similarly we can assume that
there is a section
$\zeta_{L}\colon Lw\rightarrow L$
of the residue map of $w$.
Replacing further the structure
$(K,v,\zeta_K)$
by an elementary extension
we can assume that in addition it is 
$\max\{|Lw|,\aleph_0\}$-saturated,
and then so is $Kv$,
hence 
the assumption $\mathrm{Th}_{\forall\exists}(Kv)\subseteq\mathrm{Th}_{\forall\exists}(Lw)$
gives an existentially closed embedding
$\varphi_{k}\colon Lw\rightarrow Kv$
by Lemma~\ref{lem:AE_abstract}.

By mapping $t\mapsto\pi_{K}$,
and using the universal property of the henselization,
we extend $\zeta_{K}$
to an embedding of valued fields with distinguished uniformizers
${\zeta}_{K}'\colon (Kv(t)^{h},v_{t},t)\rightarrow(K,v,\pi_{K})$,
\new{where $Kv(t)^{h}$ denotes the henselization, as introduced above}.
By \cite[Corollary 4.14]{ADF23},
\Rfour\ implies that ${\zeta}_{K}'$ is existentially closed.
Similarly,
by mapping $t\mapsto\pi_{L}$,
we extend $\zeta_{L}$ to an embedding of valued fields with distinguished element
${\zeta}_{L}'\colon(Lw(t)^{h},v_{t},t)\rightarrow(L,w,\pi_{L})$.
Note that ${\zeta}_{L}'$ induces an isomorphism on the residue fields.
Finally, by mapping $t\mapsto t$,
we extend $\varphi_{k}$
to an embedding of valued fields with distinguished uniformizers
$\varphi\colon(Lw(t)^{h},v_{t},t)\rightarrow(Kv(t)^{h},v_{t},t)$.
Again by 
\cite[Corollary 4.14]{ADF23},
\Rfour\ implies that $\varphi$ is existentially closed.
The situation looks as follows:
$$
 	\xymatrix{
	(K,v,\pi_{K})
    \ar@{.>}@/_0.7pc/[rd]_{\rm res}
&
  	(Kv(t)^{h},v_{t},t)
  	\ar@{->}[l]_{{\zeta}_{K}'}
    \ar@{.>}@/_-0.7pc/[d]^{\rm res}
&
 	(Lw(t)^{h},v_{t},t)
   	\ar@{->}[l]_{\varphi}
  	\ar@{->}[r]^{{\zeta}_{L}'}
   \ar@{.>}@/_0.7pc/[d]_{\rm res}
	&
	(L,w,\pi_{L})
 \ar@{.>}@/_-0.7pc/[ld]^{\rm res}\\
 &
 Kv
 \ar@{->}[u]
 \ar@{->}[ul]_{\zeta_K}
 &
 Lw
 \ar@{->}[u]
  \ar@{->}[l]_{\varphi_k}
  \ar@{->}[ur]^{\zeta_L}
 &
  }
$$
By Lemma~\ref{lem:AE_abstract},
$$
 	\mathrm{Th}_{\forall^{}\exists}(K,v,\pi_{K})
 \subseteq
	\mathrm{Th}_{\forall^{}\exists}(Kv(t)^{h},v_{t},t)
\subseteq
	\mathrm{Th}_{\forall^{}\exists}(Lw(t)^{h},v_{t},t).
$$
In particular,
$\mathrm{Th}_{\forall^{\mathbf{k}}\exists}(K,v,\pi_{K})
\subseteq\mathrm{Th}_{\forall^{\mathbf{k}}\exists}(Lw(t)^{h},v_{t},t)
$
and we also have
$\mathrm{Th}_{\forall^{\mathbf{k}}\exists}(Lw(t)^{h},v_{t},t)
 \subseteq\mathrm{Th}_{\forall^{\mathbf{k}}\exists}(L,w,\pi_{L})$	
by Lemma~\ref{lem:AE_1} applied to $\zeta_L'$ with $C=\{t\}$,
which concludes case $(a)$.
In case $(b)$,
by the usual Ax--Kochen/Ershov principle for residue characteristic zero, 
$\zeta_L'$
is an elementary embedding (it is an isomorphism on residue fields, and the embedding of value groups $\mathbb{Z}=v_tLw(t)^h\rightarrow wL$ is elementary since the theory of $\mathbb{Z}$-groups is model complete in the language of ordered abelian groups with a constant symbol for the smallest positive element \cite[Theorem 4.1.3]{PD}), and so in particular ${\rm Th}_{\forall\exists}(Lw(t)^h,v_t,t)={\rm Th}_{\forall\exists}(L,w,\pi_L)$.
\end{proof}

\begin{question}
Does Proposition \ref{prof:AE_char0}$(b)$ hold similarly for all the ``standard'' fragments
$\exists,\forall\exists,\exists\forall\exists,\dots$?
Can this be extended to a full Ax--Kochen/Ershov-principle of the following form:
If $(K,v)$ and $(L,w)$ are equicharacteristic zero henselian nontrivially valued fields with
${\rm Th}_{\forall\exists}(Kv)\subseteq{\rm Th}_{\forall\exists}(Lw)$ and
${\rm Th}_{\forall\exists}(vK)\subseteq{\rm Th}_{\forall\exists}(wL)$,
then
${\rm Th}_{\forall\exists}(K,v)\subseteq{\rm Th}_{\forall\exists}(L,w)$,
and similarly for the other standard fragments?
\end{question}

\begin{corollary}\label{lem:AE_4}
Suppose \Rfour\ holds.
Let $(K,v),(L,w)\models\bHen$
with $\mathrm{Th}_{\forall\exists}(Kv)\subseteq\mathrm{Th}_{\forall\exists}(Lw)$.
Then  $\mathrm{Th}_{\forall^{\mathbf{k}}\exists}(K,v)\subseteq\mathrm{Th}_{\forall^{\mathbf{k}}\exists}(L,w)$. 
\end{corollary}

\begin{proof} 
There exists an extension
$(K',v')\supseteq(K,v)$
and $\pi_{K}\in K'$
such that $(K',v',\pi_{K})\models\bHepi$
and the induced embedding
$Kv\rightarrow K'v'$ is an isomorphism,
see for example the first sentence in the proof of \cite[Corollary 4.16]{ADF23}.
Lemma \ref{lem:AE_1} with $C=\emptyset$ gives that
$\mathrm{Th}_{\forall^{\mathbf{k}}\exists}(K,v)\subseteq
\mathrm{Th}_{\forall^{\mathbf{k}}\exists}(K',v')$.
Choose any $\pi_L\in L^\times$ with $w(\pi_L)>0$.
As ${\rm Th}_{\forall\exists}(K'v')={\rm Th}_{\forall\exists}(Kv)\subseteq{\rm Th}_{\forall\exists}(Lw)$,
Proposition \ref{lem:AE_3}$(a)$
implies
$\mathrm{Th}_{\forall^{\mathbf{k}}\exists}(K',v',\pi_K)\subseteq\mathrm{Th}_{\forall^{\mathbf{k}}\exists}(L,w,\pi_L)$,
so in particular 
$\mathrm{Th}_{\forall^{\mathbf{k}}\exists}(K',v')\subseteq\mathrm{Th}_{\forall^{\mathbf{k}}\exists}(L,w)$.
\end{proof}

\begin{corollary}\label{cor:reductions}
\begin{enumerate}[$(a)$]
   \item If \Rfour\ holds, there exists a computable map
    $\tau\colon{\rm Sent}_{{\forall^{\mathbf{k}}\exists}}(\mathfrak{L}_{\rm val}(\varpi))\rightarrow{\rm Sent}_{\forall\exists}(\mathfrak{L}_{\rm ring})$
    such that for every $(K,v,\pi)\models\bHepi$ we have
    ${\rm Th}_{{\forall^{\mathbf{k}}\exists}}(K,v,\pi)=\tau^{-1}({\rm Th}_{\forall\exists}(Kv))$.        \item There exists a computable map
    $\tau\colon{\rm Sent}_{\forall\exists}(\mathfrak{L}_{\rm val}(\varpi))\rightarrow{\rm Sent}_{\forall\exists}(\mathfrak{L}_{\rm ring})$
    such that for every $(K,v,\pi)\models\bHepiZ_0$ we have
    ${\rm Th}_{\forall\exists}(K,v,\pi)=\tau^{-1}({\rm Th}_{\forall\exists}(Kv))$.
    \item If \Rfour\ holds, there exists a computable map
    $\tau\colon{\rm Sent}_{{\forall^{\mathbf{k}}\exists}}(\mathfrak{L}_{\rm val})\rightarrow{\rm Sent}_{\forall\exists}(\mathfrak{L}_{\rm ring})$
    such that for every $(K,v)\models\bHen$ we have
    ${\rm Th}_{{\forall^{\mathbf{k}}\exists}}(K,v)=\tau^{-1}({\rm Th}_{\forall\exists}(Kv))$.        
\end{enumerate}
\end{corollary}

\begin{proof}
Parts $(a)$ and $(b)$ are consequences of $(a)$ and $(b)$ of Proposition \ref{prof:AE_char0},
and part $(c)$ is a consequence of
Corollary \ref{lem:AE_4},
where in each case the map $\tau$
can be obtained using \cite[Corollary 2.23]{AF23}.
\end{proof}

\begin{remark}\label{rem:AE}
As a consequence of Corollary \ref{lem:AE_4},
if $(K,v),(L,w)\models\bHen$
and ${\rm Th}_{\forall\exists}(Kv)={\rm Th}_{\forall\exists}(Lw)$, then
${\rm Th}_{\forall^{\mathbf{k}}\exists}(K,v)
={\rm Th}_{\forall^{\mathbf{k}}\exists}(L,w)$ assuming \Rfour.
Under the stronger assumption
${\rm Th}(Kv)={\rm Th}(Lw)$
and in the special case where
$Kv$ (and thus $Lw$) is perfect,
this  ``$\forall^{\mathbf{k}}\exists$-completeness'' was shown in 
\cite[Corollary 5.7]{AF16}
without using \Rfour.

One can show unconditionally that if
$k\preccurlyeq_\exists l$,  then  $(k(t)^h,v_t)\preccurlyeq_\exists(l(t)^h,v_t)$.
That is, the second application of 
\cite[Corollary 4.14]{ADF23}
in the proof of Proposition \ref{lem:AE_3}
can be replaced by a direct argument that does not use \Rfour.

As a special case of Proposition \ref{lem:AE_3}$(a)$,
assuming \Rfour\ we get that 
\begin{equation}\label{eqn:kthktt}
 \mathrm{Th}_{\forall^{\mathbf{k}}\exists}(k(t)^{h},v_{t})=\mathrm{Th}_{\forall^{\mathbf{k}}\exists}(k(\!(t)\!),v_{t}).
\end{equation}
This can be improved in two ways:
Firstly, by a result of Ershov (see e.g.~\cite[Lemma 4.5]{AF16}), we have
$(k(t)^{h},v_{t})\preccurlyeq_{\exists}(k(\!(t)\!),v_{t})$,
and so
$\mathrm{Th}_{\forall\exists}(k(t)^{h},v_{t})\supseteq\mathrm{Th}_{\forall\exists}(k(\!(t)\!),v_{t})$,
which combined with Lemma \ref{lem:AE_1}
gives (\ref{eqn:kthktt})
without assuming \Rfour.
The second improvement is given in the following lemma.
\end{remark}

\begin{lemma}\label{lem:AE_improvement}
Assume \Rfour. For every field $k$,
$\mathrm{Th}_{\forall\exists}(k(t)^{h},v_{t},t)=\mathrm{Th}_{\forall\exists}(k(\!(t)\!),v_{t},t)$.
\end{lemma}
\begin{proof}
Since $(k(t)^{h},v_{t},t)\preccurlyeq_{\exists}(k(\!(t)\!),v_{t},t)$,
there exists an elementary extension $(K,v,\pi_{v})\succcurlyeq(k(t)^{h},v_{t},t)$
and an $\mathfrak{L}_{\mathrm{val}}(\varpi)$-embedding
$(k(\!(t)\!),v_{t},t)\rightarrow(K,v,\pi_{v})$.
By \cite[Corollary 4.14]{ADF23},
since $k[\![t]\!]$ is excellent, this embedding is existentially closed.
Thus
$\mathrm{Th}_{\forall\exists}(k(t)^{h},v_{t},t)\supseteq\mathrm{Th}_{\forall\exists}(k(\!(t)\!),v_{t},t)\supseteq
\mathrm{Th}_{\forall\exists}(K,v,\pi_v)=\mathrm{Th}_{\forall\exists}(k(t)^{h},v_{t},t)$ by
Lemma \ref{lem:AE_abstract}.
\end{proof}

\begin{lemma}\label{lem:AE_kth}
Assume \Rfour\
and let $k,k'$ be fields.
If $k\equiv k'$ then
${\rm Th}_{\forall\exists}(k(t)^{h},v_{t},t)={\rm Th}_{\forall\exists}(k'(t)^{h},v_{t},t)$.
\end{lemma}
\begin{proof}
By the Keisler--Shelah theorem \cite[Theorem 8.5.10]{Hodges}
it suffices to show that
$(k(t)^{h},v_{t},t)\equiv_{\forall\exists}(k^{\mathcal{U}}(t)^{h},v_{t},t)$
for every ultrapower $k^\mathcal{U}$ of $k$.
The embedding $k\rightarrow k^\mathcal{U}$
extends uniquely to an embedding 
$(k(t)^h,v_t,t)\rightarrow(k^\mathcal{U}(t)^h,v_t,t)$.
The inclusion $k\rightarrow k(t)^h$ on the other hand
extends to an embedding
$k^\mathcal{U}\rightarrow (k(t)^h)^\mathcal{U}$
and then further to an embedding
$f\colon(k^\mathcal{U}(t)^h,v_t,t)\rightarrow ((k(t)^h)^\mathcal{U},v_t^\mathcal{U},t^\mathcal{U})=(k(t)^h,v_t,t)^\mathcal{U}$
over $k(t)^h$,
which immediately shows that
$(k(t)^{h},v_t)\preccurlyeq_{\exists}(k^{\mathcal{U}}(t)^{h},v_t)$,
hence
${\rm Th}_{\forall\exists}(k^\mathcal{U}(t)^{h},v_{t},t)\subseteq{\rm Th}_{\forall\exists}(k(t)^{h},v_{t},t)$
by Lemma \ref{lem:AE_abstract}.
Moreover, the embedding $f$ induces an isomorphism of residue fields
$k^\mathcal{U}(t)^hv_t=k^\mathcal{U}= (k(t)^h)^\mathcal{U}v_t^\mathcal{U}$,
hence $f$ is existentially closed by 
\cite[Corollary 4.14]{ADF23},
and therefore
also
${\rm Th}_{\forall\exists}(k^\mathcal{U}(t)^{h},v_{t},t)\supseteq{\rm Th}_{\forall\exists}(k(t)^{h},v_{t},t)$,
again by Lemma \ref{lem:AE_abstract}.
\end{proof}

\begin{lemma}\label{lem:AEth}
Assume \Rfour. Let $k,l$ be fields.
If $\mathrm{Th}_{\forall\exists}(k)\subseteq\mathrm{Th}_{\forall\exists}(l)$,
then $\mathrm{Th}_{\forall\exists}(k(t)^{h},v_{t},t)\subseteq\mathrm{Th}_{\forall\exists}(l(t)^{h},v_{t},t)$.
\end{lemma}

\begin{proof}
By Corollary \ref{cor:AE}
there exists $l\preccurlyeq l^*$ such that for every finite $B\subseteq l^*$ there exists $k_B\succcurlyeq k$ and an $\mathfrak{L}_{\rm ring}$-embedding $f_B\colon k_B\rightarrow l^*$ with $B\subseteq f_B(k_B)$.
By Lemma \ref{lem:AE_kth},
$(l^{*}(t)^h,v_t,t)\equiv_{\forall\exists}(l(t)^h,v_t,t)$,
so it suffices to verify Lemma \ref{lem:AnE_abstract}$(e)$ for $M=(k(t)^h,v_t,t)$, $N=(l^*(t)^h,v_t,t)$,
and $N^*=N$.
So let $b_1,\dots,b_n\in l^{*}(t)^{h}$.
As $l^*(t)^h$ is the direct limit of
$l'(t)^h$ for $l'$ running over finitely generated subfields of $l^*$, there exists a finite $B\subseteq l^*$ such that
$b_1,\dots,b_n\in \mathbb{F}(B)(t)^h$,
where $\mathbb{F}$ is the prime field of $l$.
The map $f_B$ extends to an $\mathfrak{L}_{\rm val}(\varpi)$-embedding
$f\colon(k_{B}(t)^h,v_{t},t)\rightarrow(l^*(t)^h,v_{t},t)$,
and $b_1,\dots,b_n\in f(k_{B}(t)^h)$.
So since
$(k_{B}(t)^{h},v_{t},t)\equiv_{\forall\exists}(k(t)^{h},v_{t},t)=M$
by Lemma \ref{lem:AE_kth},
$M_{\underline{b}}:=(k_{B}(t)^h,v_t,t)$ 
satisfies Lemma \ref{lem:AnE_abstract}$(e)$.
\end{proof}

\begin{proposition}\label{prop:AE_Hepi}
Assume \Rfour.
Let $(K,v,\pi_{K}),(L,w,\pi_{L})\models\bHepi$
such that $wL\cong\mathbb{Z}$ and 
$\mathcal{O}_w$ is excellent.
If $\mathrm{Th}_{\forall\exists}(Kv)\subseteq\mathrm{Th}_{\forall\exists}(Lw)$,
then $\mathrm{Th}_{\forall\exists}(K,v,\pi_{K})\subseteq\mathrm{Th}_{\forall\exists}(L,w,\pi_{L})$.
\end{proposition}

\begin{proof}
Let $k=Kv$ and $l=Lw$.
By \cite[Proposition 4.5]{ADF23},
replacing $(K,v,\pi_K)$ by an elementary extension,
there exists a section
$k\rightarrow K$ of the residue map of $v$,
which extends to an
$\mathfrak{L}_{\mathrm{val}}(\varpi)$-embedding
$(k(t)^{h},v_{t},t)\rightarrow(K,v,\pi_{K})$.
By \cite[Corollary 4.14]{ADF23},
since $\mathcal{O}_{v_t}$ is excellent,
this embedding is existentially closed.
Therefore
$\mathrm{Th}_{\forall\exists}(K,v,\pi_{K})\subseteq\mathrm{Th}_{\forall\exists}(k(t)^{h},v_{t},t)$ 
by Lemma \ref{lem:AE_abstract}.
By Lemma \ref{lem:AEth} and Lemma~\ref{lem:AE_improvement},
$$
 \mathrm{Th}_{\forall\exists}(k(t)^{h},v_{t},t)\subseteq\mathrm{Th}_{\forall\exists}(l(t)^{h},v_{t},t)=\mathrm{Th}_{\forall\exists}(l(\!(t)\!),v_{t},t).
$$ 
Since $wL\cong\mathbb{Z}$, there is also an $\mathfrak{L}_{\mathrm{val}}(\varpi)$-embedding
$(L,w,\pi_{L})\rightarrow(l(\!(t)\!),v_{t},t)$,
obtained by completing,
see \cite[Propositions 4.2 and 4.3]{ADF23}.
By \cite[Corollary 4.14]{ADF23},
since $\mathcal{O}_w$ is excellent, 
this embedding is existentially closed.
Therefore
$\mathrm{Th}_{\forall\exists}(l(\!(t)\!),v_{t},t)\subseteq\mathrm{Th}_{\forall\exists}(L,w,\pi_{L})$,
again by Lemma \ref{lem:AE_abstract}.
\end{proof}

\begin{remark}
We stress that here we do not obtain computable reductions like in Corollary \ref{cor:reductions},
since the condition $wL\cong\mathbb{Z}$
in Proposition \ref{prop:AE_Hepi} is not elementary.
\end{remark}

\begin{lemma}\label{lem:A1E_kth}
Assume \Rfour.
Let $k,l$ be fields.
If $\mathrm{Th}_{\forall_{1}\exists}(k)\subseteq\mathrm{Th}_{\forall_{1}\exists}(l)$,
then $\mathrm{Th}_{\forall_{1}\exists}(k(t)^{h},v_{t})\subseteq\mathrm{Th}_{\forall_{1}\exists}(l(t)^{h},v_{t})$.
\end{lemma}

\begin{proof}
By Lemma \ref{lem:AnE_abstract}$(d)$,
there exists $l^{*}\succcurlyeq l$ such that for every $b\in l^{*}$ there exists $k_{b}\succcurlyeq k$, $a_b\in k_b$ and an $\mathfrak{L}_{\mathrm{ring}}$-embedding
$f_{b}\colon k_{b}\rightarrow l^{*}$ 
with $f_b(a_b)=b$.
By Lemma \ref{lem:AE_kth},
$(l^{*}(t)^h,v_t)\equiv_{\forall\exists}(l(t)^h,v_t)$,
so it suffices to verify Lemma \ref{lem:AnE_abstract}$(e)$ for $M=(k(t)^h,v_t)$,
$N=(l^*(t)^h,v_t)$ and $n=1$, with $N^*:=N$.

So let $c\in l^{*}(t)^{h}$.
Without loss of generality, $v_t(c)\geq0$,
so $c=b+u$ for $b\in l^{*}$ and $v_t(u)>0$.
If $u=0$, then $c=b\in l^{*}$,
so
$f_{b}$ may be extended to an $\mathfrak{L}_{\mathrm{val}}(\varpi)$-embedding
$f_b'\colon(k_{b}(t)^{h},v_{t},t)\rightarrow(l^{*}(t)^{h},v_{t},t)$,
and $f_b'(a_b)=f_b(a_b)=b=c$.
Otherwise, if $u\neq 0$, 
then $u$ is transcendental over $l^*$
and $v_t|_{l^*(u)}$ is the $u$-adic valuation on $l^*(u)$, 
so we may extend $f_{b}$ to an $\mathfrak{L}_{\mathrm{val}}$-embedding 
$f_{b}''\colon(k_{b}(t)^{h},v_{t})\rightarrow(l^{*}(t)^{h},v_{t})$
with $f_b''(t)=u$.
Then $f_{b}''(a_{b}+t)=f_{b}(a_{b})+f_{b}(t)=b+u=c$.

In both cases,
$(k_{b}(t)^{h},v_{t})\equiv_{\forall\exists}(k(t)^{h},v_{t})$
by Lemma \ref{lem:AE_kth},
in particular
$(k_{b}(t)^{h},v_{t})\equiv_{\forall_{1}\exists}(k(t)^{h},v_{t})$.
Therefore, 
$M_b:=(k_b(t)^h,v_t)$ 
satisfies Lemma \ref{lem:AnE_abstract}$(e)$.
\end{proof}

\begin{proposition}\label{prop:HediscA1E}
Assume \Rfour.
Let $(K,v),(L,w)\models\bHed$.
If ${\rm Th}_{\forall_{1}\exists}(Kv)\subseteq{\rm Th}_{\forall_{1}\exists}(Lw)$,
then
${\rm Th}_{\forall_{1}\exists}(K,v)\subseteq{\rm Th}_{\forall_{1}\exists}(L,w)$.
\end{proposition}

\begin{proof}
Let $(K,v),(L,w)\models\bHed$ and suppose that
$\mathrm{Th}_{\forall_{1}\exists}(k)\subseteq\mathrm{Th}_{\forall_{1}\exists}(l)$,
where $k=Kv$ and $l=Lw$.
Let $\pi_{K}$ be any uniformizer of $v$,
so that $(K,v,\pi_{K})\models\bHepi$.
As above (for example, apply Proposition \ref{prop:AE_Hepi}),
$\mathrm{Th}_{\forall\exists}(K,v,\pi_{K})\subseteq\mathrm{Th}_{\forall\exists}(k(t)^{h},v_{t},t)$,
in particular
$\mathrm{Th}_{\forall_{1}\exists}(K,v)\subseteq\mathrm{Th}_{\forall_{1}\exists}(k(t)^{h},v_{t})$.
By Lemma~\ref{lem:A1E_kth},
$\mathrm{Th}_{\forall_{1}\exists}(k(t)^{h},v_{t})\subseteq\mathrm{Th}_{\forall_{1}\exists}(l(t)^{h},v_{t})$.

By \cite[Proposition 4.5]{ADF23}, replacing $(L,w)$ by an elementary extension,
there exists a section $f\colon l\rightarrow L$ of the residue map of $w$.
For every $s\in L^\times$ with $w(s)>0$,
$f$ extends uniquely to an embedding $f_s\colon(l(t)^h,v_t,t)\rightarrow (L,w,s)$,
as in the proof of Lemma \ref{lem:A1E_kth}.
By Lemma \ref{lem:AnE_abstract}$(e)$ (applied with $n=1$),
in order to show that
$\mathrm{Th}_{\forall_{1}\exists}(l(t)^{h},v_{t})\subseteq \mathrm{Th}_{\forall_{1}\exists}(L,w)$,
it suffices to find for every $c\in L$ an $\mathfrak{L}_{\mathrm{val}}$-embedding
$g\colon(l(t)^{h},v_{t})\rightarrow(L,w)$
with $c\in g(l(t)^{h})$.
Replacing $c$ by $c^{-1}$ if necessary,
we may assume that $w(c)\geq0$,
and then we may write $c=f(b)+u$ with $b\in l$ and $w(u)>0$.
If $u=0$,
then $c=f(b)$,
and we 
let $g=f_s$ where $s\in L^\times$ with $w(s)>0$ is arbitrary,
and then $c=g(b)$.
Otherwise we 
let $g=f_u$,
and then $c=g(b+t)$
\end{proof}

\begin{corollary}\label{cor:HediscA1E}
If \Rfour\ holds, there exists a computable map
    $\tau\colon{\rm Sent}_{{\forall_1\exists}}(\mathfrak{L}_{\rm val})\rightarrow{\rm Sent}_{\forall_1\exists}(\mathfrak{L}_{\rm ring})$
    such that for every $(K,v)\models\bHed$ we have
    ${\rm Th}_{{\forall_1\exists}}(K,v)=\tau^{-1}({\rm Th}_{\forall_1\exists}(Kv))$.
\end{corollary}

\begin{proof}
The map $\tau$ can again be obtained
from Proposition \ref{prop:HediscA1E}
by \cite[Corollary 2.23]{AF23}.
\end{proof}

\begin{remark}
We remark that the statement of Proposition \ref{prop:HediscA1E}
becomes false if we replace $\bHed$ 
\begin{enumerate}[$(a)$]
\item by
$\bHen$,
as divisibility of the value group of a model of $\bHen$ is axiomatizable by a theory $T_{\rm div}\subseteq{\rm Sent}_{\forall_1\exists}(\mathfrak{L}_{\mathrm{val}})$
(and so for example
${\rm Th}_{\forall_1\exists}(k(\!(t)\!),v_t)\neq{\rm Th}_{\forall_1\exists}(k(\!(t^\mathbb{Q})\!),v_t)$), 
\item 
by $\bHepi$ ,
as the theory of $\mathbb{Z}$-groups
for the value group of a model of $\bHepi$
is axiomatizable by a theory
$T_{\mathbb{Z}}\subseteq{\rm Sent}_{\forall_{1}\exists}(\mathfrak{L}_{\mathrm{val}}(\varpi))$
(and so for example
${\rm Th}_{\forall_1\exists}(k(\!(t)\!),v_t,t)\neq{\rm Th}_{\forall_1\exists}(k(\!(t^{\mathbb{Z}\oplus\mathbb{Z}})\!),v_t,t)$),
or 
\item 
by
$\bHepiZ$,
as $\forall z\pi_{p,2}(t,z)$ (cf.~Lemma \ref{lem:pindep}) expresses that $t$ is a $p$-basis (and so for example
${\rm Th}_{\forall_1\exists}(k(\!(t)\!),v_t,t)\neq{\rm Th}_{\forall_1\exists}(k(t,s)^h,v_t,t)$,
where $s\in k(\!(t)\!)\setminus k(t)^h$).
This also shows that the proposition becomes false even for $\bHed$ if we replace $\forall_1\exists$ by $\forall_2\exists$
(take $\forall z_1,z_2\pi_{p,2}(z_1,z_2)$ instead of $\forall z\pi_{p,2}(t,z)$).
A counterexample of Kuhlmann shows 
that the proposition becomes false even if we replace
$\bHed$ by $\bHepiZ$ together with axioms for defectlessness, see~\cite[Theorem 1.3]{Kuh01}. 
\end{enumerate}
Note that in all of these examples, actually $Kv=Lw$ rather than just ${\rm Th}_{\forall_1\exists}(Kv)\subseteq{\rm Th}_{\forall_1\exists}(Lw)$.
\end{remark}

\begin{remark}
The work \cite{DF24} studies the common existential theory of all or almost all completions of a global function field.
Proposition \ref{prop:HediscA1E} allows to
extend some of the results there from $\exists$ to $\forall_1\exists$,
assuming \Rfour.
More precisely, Proposition 3.2, Theorem 3.4 and Corollary 3.5 of \cite{DF24} remain valid if one replaces ``universal/existential''
by ``$\forall_1\exists/\exists_1\forall$''
(denoting the smallest fragment containing $\forall_1\exists$ closed under negation)
and adding the assumption \Rfour.
In particular,
for a given global function field $K$ with family of completions $(\hat{K}_v,\hat{v})_{v\in\mathbb{P}_K}$,
both $\bigcap_{v\in\mathbb{P}_K}{\rm Th}_{\forall_1\exists}(\hat{K}_v,\hat{v})$ and
$\bigcup_{P\subseteq\mathbb{P}_K{\rm finite}}\bigcap_{v\in\mathbb{P}_K\setminus P}{\rm Th}_{\forall_1\exists}(\hat{K}_v,\hat{v})$ are decidable
assuming \Rfour.
\end{remark}

The following two propositions are an elaboration of \cite[Remark 7.9]{AF16}.

\begin{lemma}\label{lem:AutLaurent}
Let $k$ be a field and $K=k(\!(t)\!)$. Then
${\rm Aut}(K/k)\subseteq{\rm Aut}(K,v_t)$,
$\{\sigma(t):\sigma\in{\rm Aut}(K/k)\}=\mathfrak{m}_{v_t}\setminus\mathfrak{m}_{v_t}^2$, the set of uniformizers of $v_t$.
Similarly,
${\rm End}(K/k)\subseteq{\rm End}(K,v_t)$
and
$\{\sigma(t):\sigma\in{\rm End}(K/k)\}=\mathfrak{m}_{v_t}\setminus\{0\}$.
\end{lemma}

\begin{proof}
The claim that each $\sigma\in{\rm Aut}(K/k)$ preserves $v_t$
is contained in \cite{Schilling},
and also follows from \cite{Ax}.
By \cite[Theorem 1]{Schilling},
$\sigma\mapsto\sigma(t)/t$ gives a bijection ${\rm Aut}(K/k)\rightarrow\mathcal{O}_v^\times$, which implies that
$\{\sigma(t):\sigma\in{\rm Aut}(K/k)\}$ is the set of uniformizers.
For each $0\neq s\in\mathfrak{m}_t$, $\sum_ia_it^i\mapsto\sum_ia_is^i$ defines a $\sigma\in{\rm End}(K/k)$ with $\sigma(t)=s$.
Conversely, if $\sigma\in{\rm End}(K/k)$, then
$L:=\sigma(K)$ carries the henselian valuation $w:=v_t\circ\sigma^{-1}$,
whose extension to $F:=\overline{L}\cap K$ must be equivalent to
$v_t|_F$ by F.~K.~Schmidt's theorem \cite[Theorem 4.4.1]{EP},
so since $v_t(t)>0$ also $w(t)>0$ and therefore $v_t(\sigma(t))>0$.
This also shows that $\sigma\in{\rm End}(K,v_t)$.
\end{proof}

\begin{lemma}\label{lem:uniformizer}
There is $\nu(x)\in\Form_{\forall_1}(\mathfrak{L}_{\rm val})$ which in every valued field $(K,v)$ defines the set of uniformizers.
\end{lemma}

\begin{proof}
Let $\nu(x)$ be $\forall z( v(x)>0 \wedge (v(z)\geq v(x)\vee v(z)\leq 0))$.
\end{proof}

\begin{proposition}\label{lem:AnE}
Let $\mathsf{F}\in\{\mathsf{F}_0,\exists,\forall,\exists\forall,\forall\exists,\dots\}$.
There is a computable map 
$$
 \tau\colon{\rm Form}_{\forall_n[\exists_1[\mathsf{F}]]}(\mathfrak{L}_{\rm val}(t))\rightarrow{\rm Form}_{\forall_{n+1}[\exists_1[\mathsf{F}]]}(\mathfrak{L}_{\rm val})
$$
such that for every field $k$ and every
$\varphi\in{\rm Form}_{\forall_n[\exists_1[\mathsf{F}]]}(\mathfrak{L}_{\rm val}(t))$ we have
${\rm Var}(\varphi)={\rm Var}(\tau\varphi)$ and
$$
 \varphi((k(\!(t)\!),v_t,t))\cap k^m = \tau\varphi((k(\!(t)\!),v_t))\cap k^m.
$$
In particular,
${\rm Th}_{\forall_n\exists}(k(\!(t)\!),v_t,t)=\tau^{-1}({\rm Th}_{\forall_{n+1}\exists}(k(\!(t)\!),v_t))$.
\end{proposition}

\begin{proof}
We write $K$ for the $\mathfrak{L}_{\rm val}$-structure $(k(\!(t)\!),v_t)$.
Write the formula $\nu$ from Lemma \ref{lem:uniformizer}  as
$\nu(x)=\forall z\eta(x,z)$ with $\eta\in{\rm Form}(\mathfrak{L}_{\rm val})$ quantifier-free.
Given $\varphi(\underline{u})\in{\rm Form}_{\forall_n[\exists_1[\mathsf{F}]]}(\mathfrak{L}_{\rm val}(t))$,
write $\varphi=\forall y_1,\dots,y_n\exists z\psi(\underline{u},t,\underline{y},z)$
with $\psi(\underline{u},w,\underline{y},z)\in{\rm Form}_{\mathsf{F}}(\mathfrak{L}_{\rm val})$,
and define 
$$
 \tau\varphi \;:=\; \forall x\forall y_1,\dots,y_n \exists z (\neg\eta(x,z)\vee\psi(\underline{u},x,\underline{y},z)).
$$
Then $\tau\varphi\in{\rm Form}_{\forall_{n+1}[\exists_1[\mathsf{F}]]}(\mathfrak{L}_{\rm val})$,
as $\neg\eta\in{\rm Form}_{\mathsf{F}_0}(\mathfrak{L}_{\rm val})\subseteq{\rm Form}_{\mathsf{F}}(\mathfrak{L}_{\rm val}) $ and
$\psi\in{\rm Form}_{\mathsf{F}}(\mathfrak{L}_{\rm val})$.
If $\underline{a}\in \tau\varphi(K)\cap k^m$,
then in particular $K\models\forall\underline{y}\exists z(\neg\eta(t,z)\vee\psi(\underline{a},t,\underline{y},z))$, 
and since $t$ is a uniformizer and therefore $K\models\forall z\eta(t,z)$
we get that $K\models\forall\underline{y}\exists z\psi(\underline{a},t,\underline{y},z)$,
hence $\underline{a}\in\varphi(K)$.
Conversely, if $\underline{a}\in\varphi(K)$,
then $K\models\forall\underline{y}\exists z\psi(\underline{a},t,\underline{y},z)$.
Given any $x\in k(\!(t)\!)$, either $x$ is a uniformizer,
in which case there is $\sigma\in{\rm Aut}(K)$ with $\sigma(t)=x$
and $\sigma|_k={\rm id}_k$ (Lemma~\ref{lem:AutLaurent})
and therefore
$K\models\forall\underline{y}\exists z\psi(\underline{a},x,\underline{y},z)$,
or $x$ is not a uniformizer, in which case there exists $z\in k(\!(t)\!)$ with $K\models\neg\eta(x,z)$.
Thus $K\models\forall x\forall\underline{y}\exists z(\neg\eta(x,z)\vee\psi(\underline{a},x,\underline{y},z))$, i.e.~$\underline{a}\in\tau\varphi(K)$.
\end{proof} 

\begin{proposition}\label{prop:A1Etot}
Let $\mathfrak{L}$ be a sublanguage of $\mathfrak{L}_{\rm val}$
containing the full $\mathfrak{L}_{\rm ring}$ on the field sort,
and let $q$ be a prime power.
There is a map  
$$
 \tau_\mathfrak{L}\colon{\rm Sent}_{\forall_1\exists}(\mathfrak{L})\rightarrow{\rm Sent}_{\exists}(\mathfrak{L}(t))
$$
such that for every finite field $k$ with $|k|=q$
and $\varphi\in{\rm Sent}_{\forall_1\exists}(\mathfrak{L})$
we have
$$
 K\models\varphi\quad\Leftrightarrow\quad
 K\models\tau_\mathfrak{L}\varphi,
$$ 
where $K$ is 
the $\mathfrak{L}$- respectively $\mathfrak{L}(t)$-reduct of the
$\mathfrak{L}_{\rm val}(t)$-structure $(k(\!(t)\!),v_t,t)$.
In particular,
${\rm Th}_{\forall_1\exists}(k(\!(t)\!),v_t)=\tau^{-1}_{\mathfrak{L}_{\rm val}}({\rm Th}_{\exists}(k(\!(t)\!),v_t,t))$.
The map $\tau$ is uniform in $\mathfrak{L}$
and computable uniformly in $\mathfrak{L}$ and $q$.
\end{proposition}

\begin{proof}
It suffices to consider $\varphi\in{\rm Sent}_{\forall_1[\exists]}(\mathfrak{L})$,
i.e.~$\varphi=\forall x\psi(x)$ with $\psi\in{\rm Form}_\exists(\mathfrak{L})$. 
Let
$$
 \eta_q(z_1,\dots,z_q) \;:=\; \bigwedge_{j=1}^q z_j^q=z_j\wedge\bigwedge_{i\neq j}z_i\neq z_j 
$$
and
$$
 \tau_\mathfrak{L}\varphi \;:=\; \exists y ,z_1\dots,z_q
 \bigg(\eta_q(\underline{z})\wedge yt=1\wedge\psi(y)\wedge\bigwedge_{j=1}^q\psi(z_j+t)\wedge\bigwedge_{j=1}^q\psi(z_j)\bigg).
$$
Then $\tau_\mathfrak{L}\varphi\in{\rm Sent}_\exists(\mathfrak{L}(t))$.
If $K\models\varphi$, then 
$K\models\psi(a)$ for every $a\in k(\!(t)\!)$,
and setting $y:=t^{-1}$ and letting $z_1,\dots,z_q$ be an enumeration of $k$
we have $K\models\eta_q(\underline{z})\wedge yt=1$,
so $K\models\tau_\mathfrak{L}\varphi$.
Conversely, if $K\models\tau_\mathfrak{L}\varphi$,
we obtain $\underline{b}\in k(\!(t)\!)^q$ with $k(\!(t)\!)\models\eta_q(\underline{b})$,
so $\underline{b}$ is an enumeration of $k$, 
and $\psi$ holds in $K$ for each of the elements of $k$,
for $t^{-1}$, and for all elements of the form $b+t$ with $b\in k$.
We claim that then in fact it holds for every $a\in k(\!(t)\!)$.
Indeed, if $a\notin k$ and $v_t(a)\geq 0$,
let $b\in k$ with $v_t(a-b)>0$.
By Lemma \ref{lem:AutLaurent} there exists $\sigma\in{\rm End}(K)$ with 
$\sigma|_k={\rm id}_k$ and
$\sigma(t)=a-b$,
so $K\models\psi(b+t)$ implies
$\sigma(K)\models\psi(b+(a-b))$,
which since $\psi\in{\rm Form}_\exists(\mathfrak{L})$ implies
that $K\models\psi(a)$.
And if $v_t(a)<0$,
by Lemma \ref{lem:AutLaurent} there exists $\sigma\in{\rm End}(K)$ with 
$\sigma|_k={\rm id}_k$ and
$\sigma(t)=a^{-1}$,
so, arguing like above,
$K\models\psi(t^{-1})$ implies that
$K\models\psi(a)$.
\end{proof}

\begin{proposition}\label{prop:Tred_finite}
Let $q$ be a prime power.
Let $\mathsf{F}$ 
be a fragment
such that ${\rm Form}_{\mathsf{F}}(\mathfrak{L})$
contains all quantifier-free formulas in the empty language on the field sort,
for every language $\mathfrak{L}\supseteq\mathfrak{L}_{\rm val}$.
For every such $\mathfrak{L}$
there is a map
$$
 \tau_\mathfrak{L}\colon{\rm Form}_{\forall^{\mathbf{k}}\mathsf{F}}(\mathfrak{L})\rightarrow{\rm Form}_{\exists^{\mathbf{k}}[\mathsf{F}]}(\mathfrak{L})
$$
such that for 
every $\mathfrak{L}$-structure $K'=(K,v,\dots)$ with $|Kv|=q$
and every $\varphi\in{\rm Form}_{\forall^{\mathbf{k}}\mathsf{F}}(\mathfrak{L})$ we have
${\rm Var}(\tau\varphi)={\rm Var}(\varphi)$
and $K'\models(\varphi\leftrightarrow\tau_\mathfrak{L}\varphi)$.
This $\tau_\mathfrak{L}$ is uniform in such $\mathfrak{L}$,
and if $\mathsf{F}$ is computable, 
then $\tau$ is computable uniformly in $\mathfrak{L}$ and $q$.
\end{proposition}

\begin{proof}
Let $\varphi(\underline{u})\in\Form_{\forall^{\mathbf{k}}\mathsf{F}}(\mathfrak{L})$.
Then $\varphi$ is equivalent to
${\rm prnx}_{\mathsf{F},\mathfrak{L}}(\varphi)=\forall^{\mathbf{k}} x_{1},\dots,x_{r}\psi(\underline{x},\underline{u})$
for some $r$,
with $\psi(\underline{x},\underline{u})\in\Form_{\mathsf{F}}(\mathfrak{L})$,
see Remark \ref{rem:prnx}.
Let
$$
 \tau_\mathfrak{L}\varphi \;:=\;
 \exists^{\mathbf{k}} z_1,\dots,z_q
\Bigg(\bigwedge_{1\leq i<j\leq q} z_i\neq z_j\wedge\bigwedge_{j_1=1}^q\dots\bigwedge_{j_r=1}^q \psi(z_{j_1},\dots,z_{j_r},\underline{u})\Bigg).
$$
\end{proof}

We now summarise some of our results,
first by returning to
Theorem \ref{thm:intro_Laurent} from the introduction.

\begin{proof}[Proof of Theorem \ref{thm:intro_Laurent}]
Part $(a)$ is a special case of Proposition \ref{prof:AE_char0}$(b)$ 
and Corollary \ref{cor:reductions}$(b)$.
Part $(b)$ is Proposition \ref{prop:AE_Hepi}
and Corollary \ref{cor:reductions}$(a)$.
Part $(c)$ is a special case of
Proposition \ref{prop:HediscA1E}
and Corollary \ref{cor:HediscA1E}.
Part $(d)$ follows from Proposition \ref{lem:AnE},
and part $(e)$ follows from Proposition \ref{prop:A1Etot}.
Part $(f)$ follows from the case $n=1$
of 
Proposition \ref{prop:coding_param} and
Proposition \ref{prop:coding_noparam}.
\end{proof}

\begin{remark}
We summarise some consequences
for Laurent series fields in Figure \ref{diag:Tred} on p.~\pageref{diag:Tred}.
Each arrow represents a many-one reduction.
Solid arrows without label represent trivial reduction, where the theory at the source is a computable fragment of the theory at the target.
Solid arrows with label indicate the result that justifies this reduction.
Dashed arrows are reductions in special cases, where the case is included as a label.
Dotted arrows are reductions conditional on \Rfour.
Nodes of the same colour are many-one equivalent in the special case where $k$ is finite.
\end{remark}

\begin{remark}
In this work we did not address the question of the many-one equivalence of
${\rm Th}_\exists(k(\!(t)\!),v_t)$ and
${\rm Th}_\exists(k(\!(t)\!))$, 
or similarly of
${\rm Th}_{\forall\exists}(k(\!(t)\!),v_t)$ and
${\rm Th}_{\forall\exists}(k(\!(t)\!))$. 
Here it is crucial whether the valuation ring
$k[\![t]\!]$ is $\emptyset$-definable in the $\mathfrak{L}_{\rm ring}$-structure $k(\!(t)\!)$
by existential and universal,
respectively by universal-existential and existential-universal formulas.
For work on such problems we refer the reader to
\cite{AF17} and \cite[\S3]{FehmJahnke}.
\end{remark}

We end this work with one negative result.
For perfect $k$ of positive characteristic,
combining Propositions \ref{prop:coding_param} and \ref{lem:AnE}
we obtain a rather simple
many-one reduction of
${\rm Th}_{\forall\exists}(k(\!(t)\!),v_t,t)$
to 
${\rm Th}_{\forall_2\exists}(k(\!(t)\!),v_t)$.
As the intermediate reduction from 
${\rm Th}_{\forall\exists}(k(\!(t)\!),v_t,t)$
to
${\rm Th}_{\forall_1\exists}(k(\!(t)\!),v_t,t)$
uses the existence of an $\exists$-$\mathfrak{L}_{\rm val}(t)$-definable
surjection 
$k(\!(t)\!)\rightarrow k(\!(t)\!)\times k(\!(t)\!)$,
the following suggests that
there is no similarly simple many-one reduction
to ${\rm Th}_{\forall_1\exists}(k(\!(t)\!),v_t)$.

\begin{proposition}
    For a perfect field $k$, there is no 
    $\exists$-$\mathfrak{L}_{\rm val}(k)$-definable surjection
    $k(\!(t)\!)\rightarrow k(\!(t)\!)\times k(\!(t)\!)$.
\end{proposition}

\begin{proof}
  Let $K=k(\!(t)\!)$ and suppose $f\colon K\rightarrow K\times K$
  is a map defined by 
  $\varphi(x,y_1,y_2)\in{\rm Form}_\exists(\mathfrak{L}_{\rm val}(k))$.
  We claim that if $f(x)=(t,s)$, then $s\in k(t)^h$ (the henselization with respect to $v_t$),
  which since $k(t)^h\subsetneqq K$ shows that $f$ is not surjective.

  First, suppose that $x\in k$.
  Then as $\varphi$ is an $\mathfrak{L}_{\rm val}(k)$-formula,
  applying 
  the $\sigma\in{\rm End}(K/k)$ given by $\sigma(t)=t^2$ (Lemma \ref{lem:AutLaurent})
  gives the contradiction
  $(t,s)=f(x)=f(\sigma(x))=\sigma(f(x))=(t^2,\sigma(s))$.

  Next, suppose that $x\notin k$ and $v_t(x)\geq 0$.
  Then $x=a+z$ with $a\in k$ and $v_t(z)>0$, $z\neq 0$. 
  As $k$ is perfect, \cite[Lemma 2.1]{Anscombe_onedim} gives $g\in k[X]$ and $u\in K$ with $v_t(u)=1$ such that $z=g(u)$.
  In particular, $z$  is definable
  by a formula in ${\rm Form}_\exists(\mathfrak{L}_{\rm val}(k(u))$,
  and therefore so are $t$ and $s$
  (as $\varphi\in{\rm Form}_\exists(\mathfrak{L}_{\rm val}(k))$ and $a\in k$).
  Let $\sigma\in{\rm Aut}(K/k)$ be the automorphism given by $\sigma(u)=t$ (Lemma \ref{lem:AutLaurent}).
  Then $\sigma(t)$ and $\sigma(s)$ are 
  definable in $K$ by a formula in ${\rm Form}_\exists(\mathfrak{L}_{\rm val}(k(t)))$,
  which since $(k(t)^h,v_t)\preccurlyeq_\exists (K,v_t)$ 
  (see again e.g.~\cite[Lemma 4.5]{AF16})
  implies that 
  $\sigma(t),\sigma(s) \in k(t)^h$.
  In particular, 
  $\sigma(t)$ and $\sigma(s)$
  are algebraically dependent over $k$,
  and then so are $t$ and $s$.

  Finally, suppose that $v_t(x)<0$.
  Since with $f$ also the map $x\mapsto f(x^{-1})$ (with $0^{-1}:=0$) is definable
  by a formula in
  ${\rm Form}_\exists(\mathfrak{L}_{\rm val}(k))$,
  the same argument as for the previous case works here.
\end{proof}

\new{
\appendix
\section{Computational aspects}
\label{sec:appendix}

\noindent 
In this short appendix
we discuss the notions of computability of (presentations of) languages,
sets of formulas, and functors of formulas.
}

A {\em presentation} of a (necessarily countable) language $\mathfrak{L}$
is an injection $\alpha$ mapping the symbols of $\mathfrak{L}$ to $\mathbb{N}$.
A standard Gödel coding using $\alpha$ produces an injection
$\bar{\alpha}\colon\Form(\mathfrak{L})\rightarrow\mathbb{N}$.
We denote by $\mathfrak{L}^\alpha$
the following collection of functions on $\mathbb{N}$:
the indicator function of the image of the set of function symbols, the set of relation symbols and the set of constant symbols, respectively,
and the function induced by the arity function.
We say that the presentation $\alpha$ is {\em computable}
if each of the functions in $\mathfrak{L}^{\alpha}$ is.

If $\alpha$ is a computable presentation of $\mathfrak{L}$,
a set of formulas $F\subseteq{\rm Form}(\mathfrak{L})$ is {\em decidable} if so is $\bar{\alpha}(F)$.
If we call a map $F_1\rightarrow F_2$ between
sets of formulas $F_1\subseteq{\rm Form}(\mathfrak{L}_1)$, $F_2\subseteq{\rm Form}(\mathfrak{L}_2)$
for languages $\mathfrak{L}_i$ with computable presentation $\alpha_i$
{\em computable}, we mean that the induced function $\bar{\alpha}_1(F_1)\rightarrow\bar{\alpha}_2(F_2)$ is computable,
and similarly when we speak of 
{\em many-one reductions} and {\em many-one equivalences}.
For definitions of these terms, see for example
\cite[\S11]{Post} and
\cite[Definition 1.6.8]{Soare}.

For the definition of when a functor of formulas (in the sense of Definition \ref{def:functors}) is computable, we follow the philosophy of \cite[\S3]{Miller}.
Namely, $\mathsf{F}(\mathfrak{L})$ usually has a chance to be computable only if the presentation $\alpha$ of $\mathfrak{L}$ is,
and even then the computability of $\mathsf{F}(\mathfrak{L})$ might depend on the specific presentation.
It therefore makes sense to ask for computability {\em relative} to the presentation:
There are Turing functionals%
\footnote{For the definition of {\em Turing functional} see \cite[Chapter 3.2.2]{Soare}.}
$\Phi$ and $\Psi$ such that
if $\mathfrak{L}$ has presentation $\alpha$,
then $\Phi^{\mathfrak{L}^{\alpha}}$ decides $\bar{\alpha}({\rm Form}(\mathfrak{L}))$,
and if $\mathfrak{L}_1\subseteq\mathfrak{L}_2$ have presentations $\alpha_1$ respectively $\alpha_2$,
and $g\colon{\alpha}_1(\mathfrak{L}_1)\rightarrow{\alpha}_2(\mathfrak{L}_2)$ is the induced injection,
$\Psi^{\mathfrak{L}_1^{\alpha_1}\oplus g\oplus\mathfrak{L}_2^{\alpha_2}}$ computes the induced function
$\bar\alpha_1({\rm Form}(\mathfrak{L}_1))\rightarrow\bar\alpha_2({\rm Form}(\mathfrak{L}_2))$.
A functor of formulas $\mathsf{F}$
with domain $\mathbf{L}$
is {\em computable}
if there exists a Turing functional $\Pi$ 
such that for every $\mathfrak{L}\in\mathbf{L}$ with presentation $\alpha$,
$\Pi^{\mathfrak{L}^\alpha}$ decides $\bar\alpha(\mathsf{F}(\mathfrak{L}))$.
Given functors of formulas $\mathsf{F}_1,\mathsf{F}_2$ with domain $\mathbf{L}$,
we say that a family of maps
$\tau_\mathfrak{L}\colon\mathsf{F}_1(\mathfrak{L})\rightarrow\mathsf{F}_2(\mathfrak{L})$ is
{\em uniformly computable} (in $\mathfrak{L}$)
if there exists a Turing functional $\Theta$
such that
for every $\mathfrak{L}\in\mathbf{L}$
with presentation $\alpha$,
$\Theta^{\mathfrak{L}^{\alpha}}$ computes
the induced function
$\bar\alpha(\mathsf{F}_1(\mathfrak{L}))\rightarrow\bar\alpha(\mathsf{F}_2(\mathfrak{L}))$.
When we say that 
the $\tau_\mathfrak{L}$
are computable uniformly in $\mathfrak{L}$ and some other input data,
like some $n_1,\dots,n_r\in\mathbb{N}$,
we mean that $\Theta$ computes the corresponding map
$\bar\alpha(\mathsf{F}_1(\mathfrak{L}))\times\mathbb{N}^r\rightarrow\bar\alpha(\mathsf{F}_2(\mathfrak{L}))$.
Sometimes we will also use this for 
objects $x_i$ other than $n_i\in\mathbb{N}$,
like formulas in a fixed finite language, or polynomials over a presented field,
where we then mean a suitable coding $\bar\alpha(x_i)\in\mathbb{N}$,
whose precise form is not critical
for the question of computability.

\begin{remark}
If $\mathsf{F}$ is a computable functor of formulas, 
then for every $\mathfrak{L}$ (in the domain of $\mathsf{F}$) with computable presentation $\alpha$, the set of formulas
${\rm Form}_\mathsf{F}(\mathfrak{L})$ is decidable
(in the sense that $\bar\alpha({\rm Form}_\mathsf{F}(\mathfrak{L}))$ is decidable). 
Similarly, if $\mathsf{F}_1$ and $\mathsf{F}_2$ are computable functors of formulas
and $\tau_\mathfrak{L}\colon{\rm Form}_{\mathsf{F}_1}(\mathfrak{L})\rightarrow{\rm Form}_{\mathsf{F}_2}(\mathfrak{L})$
is a uniformly computable family of maps,
then for every $\mathfrak{L}$ with computable presentation $\alpha$,
the map of formulas
$\tau_\mathfrak{L}\colon{\rm Form}_{\mathsf{F}_1}(\mathfrak{L})\rightarrow{\rm Form}_{\mathsf{F}_2}(\mathfrak{L})$ is computable
(in the sense that
$\bar\alpha({\rm Form}_{\mathsf{F}_1}(\mathfrak{L}))\rightarrow\bar\alpha({\rm Form}_{\mathsf{F}_2}(\mathfrak{L}))$ is computable).
In that case, for every  $T\subseteq {\rm Form}_{\mathsf{F}_2}(\mathfrak{L})$ we obtain a many-one reduction of
$\tau_\mathfrak{L}^{-1}(T)$ to $T$.
Such families $\tau_{\mathfrak{L}}$ occur in particular in
Propositions~\ref{prop:coding_param}, \ref{prop:coding_noparam}, \ref{prop:A1Etot}, and \ref{prop:Tred_finite}.
\end{remark}

\section*{Acknowledgements}

\noindent
The authors would like to thank Philip Dittmann for helpful comments on a preliminary version,
Russell Miller for references regarding the computability theory,
\new{and a referee for helpful suggestions regarding the presentation}.
{\color{blue}They would also like to thank Bjorn Poonen for helpful discussion regarding Proposition \ref{prop:Tred_ff_2}.}

Part of this is based upon work supported by the National Science Foundation under Grant No. DMS-1928930 while the authors participated in a program hosted by the Mathematical Sciences Research Institute in Berkeley, California, during Summer 2022.
S.~A.\ was supported by GeoMod AAPG2019 (ANR-DFG).
S.~A.\ and A.~F.\ were supported by the Institut Henri Poincaré.
A.~F.~was funded by the Deutsche Forschungsgemeinschaft (DFG) - 404427454.
{\color{blue} The simplified proof of Proposition \ref{prop:Tred_ff_2} was worked out while the authors were participating in the trimester on ``Definability, decidability, and computability'' at the Hausdorff Institute Bonn, funded by the Deutsche Forschungsgemeinschaft (DFG, German Research Foundation) under Germany‘s Excellence Strategy – EXC-2047/1 – 390685813.}

\def\bibfont{\footnotesize}
\bibliographystyle{plain}

\addresseshere

\begin{landscape}
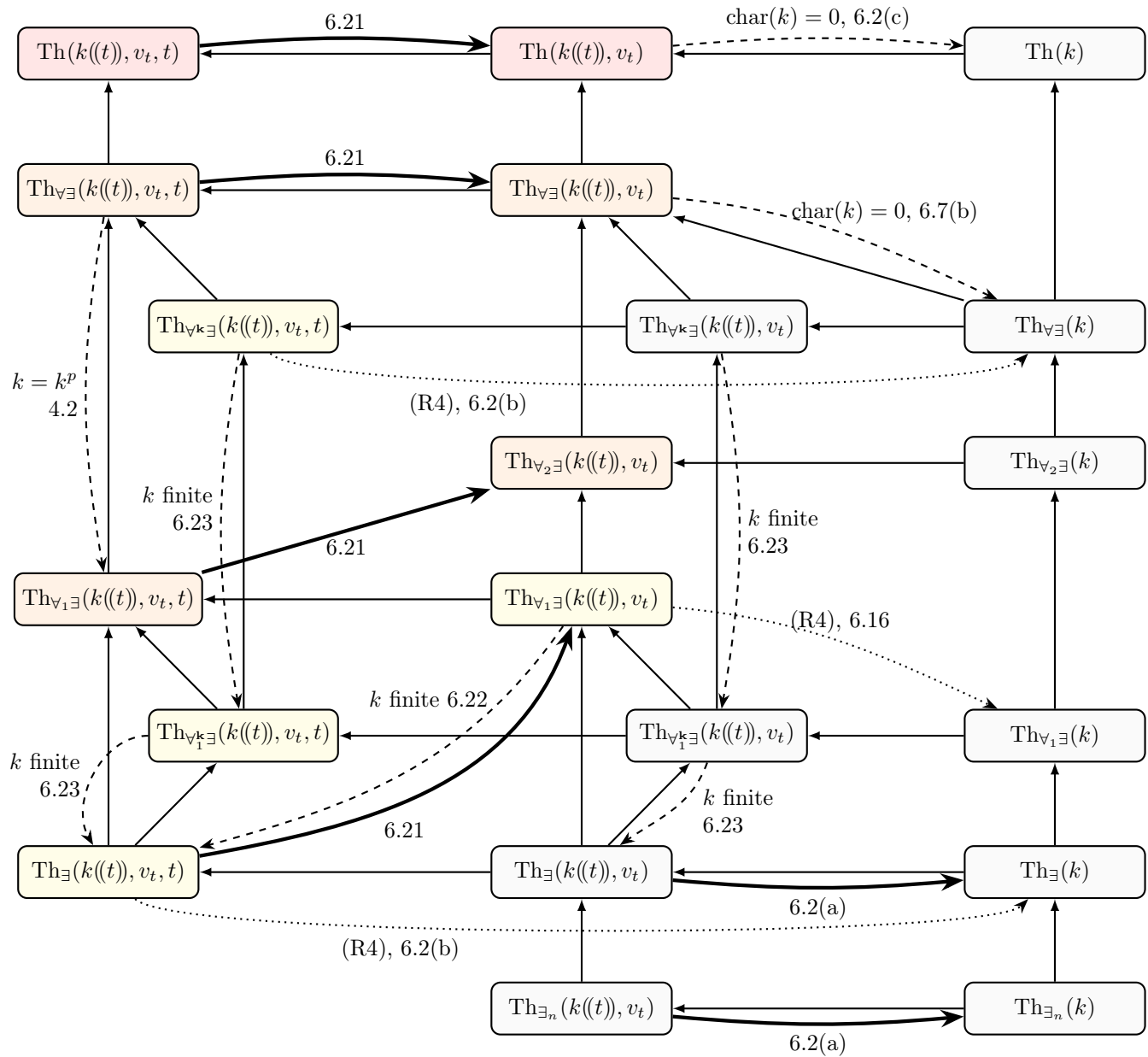
\begin{figure}[ht]
\centering

\begin{tikzpicture}[auto,scale=2.10,thick,>=latex,block/.style={draw, fill=white, rectangle, minimum height=3em, minimum width=6em},
greybox/.style={rectangle, rounded corners, minimum width=2.8cm, minimum height=0.8cm, text centered, draw=black, fill=black!02},
pinkbox/.style={rectangle, rounded corners, minimum width=2.8cm, minimum height=0.8cm, text centered, draw=black, fill=red!10},
bluebox/.style={rectangle, rounded corners, minimum width=2.8cm, minimum height=0.8cm, text centered, draw=black, fill=orange!10},
orangebox/.style={rectangle, rounded corners, minimum width=2.8cm, minimum height=0.8cm, text centered, draw=black, fill=yellow!10},
thmarrow/.style={ultra thick,-{Stealth}},
hyparrow/.style={dashed,thick,-{Stealth}},
R4arrow/.style={dotted,thick,-{Stealth}},
myarrow/.style={thick,-{Stealth}},
]

\coordinate (0) at (0,0); 
\coordinate (1) at (-3.5,0); 
\coordinate (1a) at (-2.5,0); 
\coordinate (2) at (-7.0,0); 
\coordinate (2a) at (-6.0,0); 

\coordinate (f1) at ($(0)+(0,0)$);
\coordinate (f2) at ($(0)+(0,1)$);
\coordinate (f3) at ($(0)+(0,2)$);
\coordinate (f4) at ($(0)+(0,3)$);
\coordinate (f5) at ($(0)+(0,4)$);
\coordinate (f6) at ($(0)+(0,5)$); 
\coordinate (f7) at ($(0)+(0,6)$); 
\coordinate (f8) at ($(0)+(0,7)$); 

\coordinate (vf1) at ($(1)+(0,0)$);
\coordinate (vf2) at ($(1)+(0,1)$);
\coordinate (vf3) at ($(1)+(0,2)$);
\coordinate (vf4) at ($(1)+(0,3)$);
\coordinate (vf5) at ($(1)+(0,4)$);
\coordinate (vf6) at ($(1)+(0,5)$); 
\coordinate (vf7) at ($(1)+(0,6)$); 
\coordinate (vf8) at ($(1)+(0,7)$); 
\coordinate (vfa1) at ($(1a)+(0,0)$);
\coordinate (vfa2) at ($(1a)+(0,1)$);
\coordinate (vfa3) at ($(1a)+(0,2)$);
\coordinate (vfa4) at ($(1a)+(0,3)$);
\coordinate (vfa5) at ($(1a)+(0,4)$);
\coordinate (vfa6) at ($(1a)+(0,5)$); 
\coordinate (vfa7) at ($(1a)+(0,6)$); 
\coordinate (vfa8) at ($(1a)+(0,7)$); 

\coordinate (vfpi1) at ($(2)+(0,0)$);
\coordinate (vfpi2) at ($(2)+(0,1)$);
\coordinate (vfpi3) at ($(2)+(0,2)$);
\coordinate (vfpi4) at ($(2)+(0,3)$);
\coordinate (vfpi5) at ($(2)+(0,4)$);
\coordinate (vfpi6) at ($(2)+(0,5)$); 
\coordinate (vfpi7) at ($(2)+(0,6)$); 
\coordinate (vfpi8) at ($(2)+(0,7)$); 
\coordinate (vfpia1) at ($(2a)+(0,0)$);
\coordinate (vfpia2) at ($(2a)+(0,1)$);
\coordinate (vfpia3) at ($(2a)+(0,2)$);
\coordinate (vfpia4) at ($(2a)+(0,3)$);
\coordinate (vfpia5) at ($(2a)+(0,4)$);
\coordinate (vfpia6) at ($(2a)+(0,5)$); 
\coordinate (vfpia7) at ($(2a)+(0,6)$); 
\coordinate (vfpia8) at ($(2a)+(0,7)$); 

\node[greybox] (F1) [] at (f1) {$\mathrm{Th}_{\exists_{n}}(k)$};
\node[greybox] (F2) [] at (f2) {$\mathrm{Th}_{\exists}(k)$};
\node[greybox] (F3) [] at (f3) {$\mathrm{Th}_{\forall_{1}\exists}(k)$};\node[greybox] (F5) [] at (f5) {$\mathrm{Th}_{\forall_{2}\exists}(k)$};
\node[greybox] (F6) [] at (f6) {$\mathrm{Th}_{\forall\exists}(k)$};
\node[greybox] (F8) [] at (f8) {$\mathrm{Th}_{}(k)$};

\node[greybox] (VF1) [] at (vf1) {$\mathrm{Th}_{\exists_{n}}(k(\!(t)\!),v_{t})$};
\node[greybox] (VF2) [] at (vf2) {$\mathrm{Th}_{\exists}(k(\!(t)\!),v_{t})$};
\node[greybox] (VF3) [] at (vfa3) {$\mathrm{Th}_{\forall_{1}^{\mathbf{k}}\exists}(k(\!(t)\!),v_{t})$};
\node[greybox] (VFa6) [] at (vfa6) {$\mathrm{Th}_{\forall^{\mathbf{k}}\exists}(k(\!(t)\!),v_{t})$};
\node[orangebox] (VFb4) [] at (vf4) {$\mathrm{Th}_{\forall_{1}\exists}(k(\!(t)\!),v_{t})$};
\node[bluebox] (VFb5) [] at (vf5) {$\mathrm{Th}_{\forall_{2}\exists}(k(\!(t)\!),v_{t})$};
\node[bluebox] (VF7) [] at (vf7) {$\mathrm{Th}_{\forall\exists}(k(\!(t)\!),v_{t})$};
\node[pinkbox] (VF8) [] at (vf8) {$\mathrm{Th}_{}(k(\!(t)\!),v_{t})$};

\node[orangebox] (VFPI2) [] at (vfpi2) {$\mathrm{Th}_{\exists}(k(\!(t)\!),v_{t},t)$};
\node[orangebox] (VFPI3) [] at (vfpia3) {$\mathrm{Th}_{\forall_{1}^{\mathbf{k}}\exists}(k(\!(t)\!),v_{t},t)$};
\node[orangebox] (VFPIa6) [] at (vfpia6) {$\mathrm{Th}_{\forall^{\mathbf{k}}\exists}(k(\!(t)\!),v_{t},t)$};
\node[bluebox] (VFPIb4) [] at (vfpi4) {$\mathrm{Th}_{\forall_{1}\exists}(k(\!(t)\!),v_{t},t)$};
\node[bluebox] (VFPI7) [] at (vfpi7) {$\mathrm{Th}_{\forall\exists}(k(\!(t)\!),v_{t},t)$};
\node[pinkbox] (VFPI8) [] at (vfpi8) {$\mathrm{Th}_{}(k(\!(t)\!),v_{t},t)$};

\draw[->] (F1) -- (F2);
\draw[->] (F2) -- (F3);
\draw[->] (F3) -- (F5);
\draw[->] (F5) -- (F6);
\draw[->] (F6) -- (F8);
\draw[->] (VF1) -- (VF2);
\draw[->] (VF2) -- (VF3);
\draw[->] (VF2) -- (VFb4);
\draw[->] (VF3) -- (VFa6);
\draw[->] (VF3) -- (VFb4);
\draw[->] (VFb4) -- (VFb5);
\draw[->] (VFb5) -- (VF7);
\draw[->] (VFa6) -- (VF7);
\draw[->] (VF7) -- (VF8);
\draw[->] (VFPI2) -- (VFPI3);
\draw[->] (VFPI2) -- (VFPIb4);
\draw[->] (VFPI3) -- (VFPIa6);
\draw[->] (VFPI3) -- (VFPIb4);
\draw[->] (VFPIa6) -- (VFPI7);
\draw[->] (VFPIb4) -- (VFPI7);
\draw[->] (VFPI7) -- (VFPI8);

\draw[->] (F1) -- (VF1);
\draw[->] (F2) -- (VF2);
\draw[->] (F3) -- (VF3);
\draw[->] (F5) -- (VFb5);
\draw[->] (F6) -- (VFa6);
\draw[->] (F6) -- (VF7);
\draw[->] (F8) -- (VF8);

\draw[->] (VF2) -- (VFPI2);
\draw[->] (VF3) -- (VFPI3);
\draw[->] (VFa6) -- (VFPIa6);
\draw[->] (VFb4) -- (VFPIb4);
\draw[->] (VF7) -- (VFPI7);
\draw[->] (VF8) -- (VFPI8);

\draw[thmarrow] (VF1) to[out=355,in=185] node[below] {\ref{prop:Ftt}(a)} (F1);
\draw[thmarrow] (VF2) to[out=355,in=185] node[below] {\ref{prop:Ftt}(a)} (F2);

\draw[R4arrow] (VFPI2) .. controls ($(vfpi2)+(0.5,-0.5)$) and ($(f2)+(-0.5,-0.5)$) .. node[below,shift={(-8em,0em)}] {\Rfour, \ref{prop:Ftt}(b)} (F2);
\draw[R4arrow] (VFPIa6) .. controls ($(vfpia6)+(0.5,-0.5)$) and ($(f6)+(-0.5,-0.5)$) .. node[below,shift={(-8em,0em)}] {\Rfour, \ref{prop:Ftt}(b)} (F6);

\draw[R4arrow] (VFb4)  to[out=355,in=155] node[above] {\Rfour, \ref{cor:HediscA1E}} (F3);

\draw[hyparrow] (VF7) to[out=355,in=155] node[above,shift={(2em,0em)}] {$\mathrm{char}(k)=0$, \ref{cor:reductions}(b)} (F6);
\draw[hyparrow] (VF8) to[out=005,in=175] node[above] {$\mathrm{char}(k)=0$, \ref{prop:Ftt}(c)} (F8);

\draw[hyparrow] (VFa6) to[out=280,in=080] node[right,align=left] {$k$ finite\\\ref{prop:Tred_finite}} (VF3);
\draw[hyparrow] (VF3) to[out=250,in=035] node[right,align=left,shift={(1em,0em)}] {$k$ finite\\\ref{prop:Tred_finite}} (VF2);

\draw[hyparrow] (VFPI3) to[out=180,in=120] node[left,align=right,shift={(-0.25em,0em)}] {$k$ finite\\\ref{prop:Tred_finite}} (VFPI2);
\draw[hyparrow] (VFPIa6) to[out=260,in=100] node[left,align=right,shift={(0em,1em)}] {$k$ finite\\\ref{prop:Tred_finite}} (VFPI3);
\draw[hyparrow] (VFPI7) to[out=260,in=100] node[left,align=right] {$k=k^{p}$\\\ref{prop:coding_param}} (VFPIb4);

\draw[thmarrow] (VFPIb4) to node[below] {\ref{lem:AnE}} (VFb5);
\draw[thmarrow] (VFPI7) to[out=005,in=175] node[above] {\ref{lem:AnE}} (VF7);
\draw[thmarrow] (VFPI8) to[out=005,in=175] node[above] {\ref{lem:AnE}} (VF8);

\draw[hyparrow] (VFb4) to[out=235,in=15] node[above,shift={(1.0em,2.5em)}] {$k$ finite \ref{prop:A1Etot}} (VFPI2); 
\draw[thmarrow] (VFPI2) to[out=10,in=250] node[below,shift={(-1em,-1em)}] {\ref{lem:AnE}} (VFb4);
\end{tikzpicture}

\caption{Many-one reductions between theories}
\label{diag:Tred}
\end{figure}
\end{landscape}

\end{document}